\documentclass[11pt,a4paper]{article}
\usepackage{amssymb,amsmath,amscd,amsthm,amsxtra,latexsym}
\theoremstyle{plain}

\newtheorem{theorem}{Theorem}[section]
\newtheorem{proposition}[theorem]{Proposition}
\newtheorem{lemma}[theorem]{Lemma}
\newtheorem{corollary}[theorem]{Corollary}
\newtheorem{definition}[theorem]{Definition}
\newtheorem{remark}[theorem]{Remark}
\newtheorem{example}[theorem]{Example}
\newcommand{\scrA}{\mathcal{A}}

\newcommand{\scrH}{\mathcal{H}}
\newcommand{\scrO}{\mathcal{O}}

\newcommand{\fraka}{\mathfrak{a}}
\newcommand{\frakc}{\mathfrak{c}}
\newcommand{\frakg}{\mathfrak{g}}
\newcommand{\frakh}{\mathfrak{h}}
\newcommand{\fraki}{\mathfrak{i}}
\newcommand{\frakj}{\mathfrak{j}}
\newcommand{\frakk}{\mathfrak{k}}

\newcommand{\frakp}{\mathfrak{p}}
\newcommand{\fraku}{\mathfrak{u}}

\newcommand{\rmL}{\mathrm{L}}

\newcommand{\rmO}{\mathrm{O}}
\newcommand{\rmP}{\mathrm{P}}

\newcommand{\rmS}{\mathrm{S}}

\newcommand{\rmU}{\mathrm{U}}
\newcommand{\rmZ}{\mathrm{Z}}

\newcommand{\bbC}{\mathbb{C}}
\newcommand{\bbE}{\mathbb{E}}
\newcommand{\bbF}{\mathbb{F}}

\newcommand{\bbK}{\mathbb{K}}

\newcommand{\bbR}{\mathbb{R}}

\newcommand{\bbZ}{\mathbb{Z}}

\newcommand{\R}{\bbR}
\newcommand{\C}{\bbC}

\newcommand{\Z}{\bbZ}

\newcommand{\Gl}{\mathrm{Gl}}

\newcommand{\Eig}{\mathrm{Eig}}

\newcommand{\qmq}[1]{\quad\mbox{#1}\quad}

\newcommand{\Menge}[2]{\big\{\,#1\,\big|\,#2\,\big\}}

\newcommand{\g}[2]{\langle #1,#2\rangle}

\newcommand {\Kovv}[3]{{\nabla^{\scriptscriptstyle{#1}}}_{\!#2}#3}

\newcommand{\hdisp}[3]{\overset{#2}{\underset{#1}{\parallel}}\!\!#3\,}

\newcommand{\sst}{\scriptscriptstyle}
\newcommand{\ghdisp}[4]{(\hdisp{#1}{#2}{#3})^{\sst{#4}}}
\newcommand{\so}{\mathfrak{so}}

\newcommand{\gl}{\mathfrak{gl}}

\newcommand{\osc}{\scrO}
\newcommand{\Id}{\mathrm{Id}}
\newcommand{\Iso}{\mathrm{I}}
\newcommand{\Hol}{\mathrm{Hol}}
\newcommand{\trace}{\mathrm{trace}}

\renewcommand{\i}{\mathrm{i}}

\,
\newcommand{\Spann}[2]{\big\{\,#1\,\big|\,#2\,\big\}_{\scriptstyle\R} }\,

\newcommand{\hol}{\mathfrak{hol}}
\newcommand{\fetth}{\boldsymbol{h}}

\newcommand{\Ad}{\mathrm{Ad}}
\newcommand{\Hom}{\mathrm{Hom}}

\begin{document}
%\begin{frontmatter}

\title{Parallel submanifolds with an intrinsic product structure}

\author{Tillmann Jentsch}
% \address{Institute of Mathematics, University of
%   Cologne, Weyertal 86-90, D-50931 Cologne, Germany}
% \ead{tjentsch@math.uni-koeln.de}
\date{\today}
\maketitle

\abstract{Let $M$ and $N$ be Riemannian symmetric spaces and $f:M\to N$ be a parallel
isometric immersion. We additionally assume that there exist 
simply connected, irreducible Riemannian symmetric spaces $M_i$ with
$\dim(M_i)\geq 2$ for $i=1,\ldots,r$ such that $M\cong M_1\times\cdots\times M_r$\,. 
As a starting point, we describe how the intrinsic product structure of $M$ is
reflected by a distinguished, fiberwise orthogonal direct sum
decomposition of the corresponding first normal bundle.
Then we consider the (second) osculating bundle $\osc f$\,, which is a $\nabla^N$-parallel vector subbundle of the pullback
bundle $f^*TN$\,, and establish the existence of $r$ distinguished,
pairwise commuting, $\nabla^N$-parallel vector bundle involutions on $\osc f$\,.
Consequently, the ``extrinsic holonomy Lie algebra'' of $\osc f$ bears
naturally the structure of a graded Lie algebra over the Abelian group which is given by the
direct sum of $r$ copies of $\Z/2\,\Z$\,. Our main result is the following: Provided that $N$ is of
compact or non-compact type, that $\dim(M_i)\geq 3$
for $i=1,\ldots,r$ and that none of the product slices through one 
point of $M$ gets mapped into any flat of $N$\,, we can show that $f(M)$ is a homogeneous submanifold of $N$\,.}

% \begin{keyword}
% symmetric space \sep parallel submanifold \sep homogeneous
%   submanifold \sep holonomy Lie algebra
% \MSC 53C29  \sep 53C35 \sep  53C40
% \end{keyword}

%\end{frontmatter}
\section{Introduction}
Given a Riemannian symmetric space $N$ (briefly called ``symmetric space'') and an isometric immersion  $f:M\to
N$\,,  we let $TM$\,, $TN$\,, $\bot f$\,, $h:TM\times TM\to\bot f$ and $S:TM\times\bot f\to TM$ denote the tangent bundles
of $M$ and $N$\,, the normal bundle, the second
fundamental form and the shape operator of $f$\,, respectively. Then $TM$ and
$TN$ are equipped with the Levi Civita connections $\nabla^M$ and
$\nabla^N$\,, respectively, whereas on $\bot f$ there is the usual normal connection $\nabla^\bot$
(obtained through projection).
The equations of Gau{\ss} and Weingarten state that
\begin{equation}\label{eq:GW}
\Kovv{N}{X}{Tf\,Y}=Tf(\Kovv{M}{X}{Y})+h(X,Y)\ \text{and}\ \Kovv{N}{X}{\xi}=-Tf(S_\xi(X))+\Kovv{\bot}{X}{\xi}
\end{equation}
for all $X,Y\in \Gamma(TM),\xi\in\Gamma(\bot f)$\,. On the vector bundle $\rmL^2(TM,\bot f)$ there is a connection induced by
$\nabla^M$ and $\nabla^\bot$ in a natural way, often called ``Van der
Waerden-Bortolotti connection'', characterized as follows: {\em For every curve $c:\R\to M$\,, every parallel section $b:\R\to \rmL^2(TM,\bot f)$ along $c$ and any two parallel sections
 $X,Y:\R\to TM$ along $c$\,, the curve $t\mapsto b(X(t),Y(t))$ is a parallel
 section of $\bot f$ along $c$\,.}

\sloppypar
\begin{definition}\label{de:parallel} $f$ is called
 \emph{parallel} if its second
fundamental form $h$ is a parallel section of
the vector bundle $\rmL^2(TM,\bot M)$\,.
\end{definition}

\begin{example}[``Extrinsic Circles'', see~\cite{NY}]\label{ex:circles}
A unit speed curve $c:J\to N$ is parallel if and only if it satisfies the equation
\begin{equation}
    \label{eq:circles}
\nabla^N_\partial\nabla^N_\partial\dot c=-\kappa^2 \dot c
\end{equation}
for some constant $\kappa\in\R$\,. 
\end{example}
Surprisingly, not much is known so far about parallel isometric immersions in
general. Only the special case of the ``symmetric 
submanifolds'' (in the sense of~\cite{F2},~\cite{N2} and~\cite{St}) 
of the symmetric spaces is perfectly well understood (cf.~\cite[Ch.~9.3]{BCO}). 

At a first glance, one could think that parallel submanifolds are simply the extrinsic 
analogue of symmetric spaces; remember that the latter ones are characterized by the condition $\nabla
R=0$\,. However, this comparison is a bit shortcoming; for example, not every complete parallel submanifold is 
a symmetric submanifold unless $N$ is a space form (cf.~\cite[Prop.~1]{J1}).
In fact, the discrepancy between the intrinsic and the extrinsic
situation is even wider: Whereas every symmetric space is intrinsically a homogenous space, 
complete parallel submanifolds of $N$ are not necessarily homogeneous
submanifolds (see Section~\ref{se:2.6}) unless $N$ is Euclidian or a rank-1 symmetric space.

\begin{definition}[{\cite[A.~1]{BCO}}]
\label{de:deRham}
Let $M$ be a simply connected symmetric space. 
\begin{enumerate}
\item $M$ is called
``reducible'' if it is the Riemannian product of two
Riemannian spaces of dimension at least 1, respectively; otherwise $M$ is
called ``irreducible''.
\item There exists a Euclidian space $M_0$ and simply connected, irreducible symmetric spaces $M_i$ with
$\dim(M_i)\geq 2$ for $i=1,\ldots,r$ such that $M$
is isometric to the Riemannian product $M_0\times\cdots\times M_r$\,. In this
case, the ``factors'' $M_i$ are uniquely determined (up to isometry, respectively,
and up to a permutation of $\{M_1,\ldots,M_r\}$) by means of the ``de\,Rham
Decomposition Theorem'', and $M_0\times\cdots\times M_r$ is called  the ``de\,Rham
decomposition'' of $M$\,.
\item 
In the situation of~(b), we say that $M$ has no Euclidian factor 
if $M_0$ is trivial; then $M_1\times\cdots\times M_r$ is called  the ``de\,Rham
decomposition'' of $M$\,.
\end{enumerate}
\end{definition}

Let a parallel isometric immersion $f:M\to N$ be given. It is well known that
then $M$ is necessarily a locally symmetric space (cf.~\cite[Prop.~4]{J1}); 
in fact, we can even assume that $M$ is a simply connected (globally)
symmetric space.\footnote{\label{foot:JR}
According to~\cite[Thm.~7]{JR}, for every (not
necessarily complete) parallel submanifold $M_{loc}\subset N$ there exists a simply connected
Riemannian symmetric space $M$\,, a parallel isometric immersion $f:M\to N$ and an open subset $U\subset M$\,,
such that $f|U:U\to M_{loc}$ is covering.} 
Here we consider the following more specific situation: \textbf{\boldmath
Throughout this article, we assume that $M$ and $N$ both are symmetric spaces,
that $M$ is additionally simply connected and without Euclidian factor and that
$f:M\to N$ is a parallel isometric immersion.}

\begin{example}\label{ex:ExtrinsicSplitting}
Suppose that there exist symmetric spaces $N_i$ and parallel isometric immersions
$f_i:M_i\to N_i$ from simply connected, irreducible symmetric spaces
$M_i$ with $\dim(M_i)\geq 2$ for $i=1,\ldots,r$\,. Then $M:=M_1\times\cdots\times M_r$ and
$N:=N_1\times\cdots\times N_r$ both are symmetric spaces, $M$ has no
Euclidian factor and the direct product map $f:=f_1\times\cdots\times f_r:M\to
N$ is a parallel isometric immersion, too\,.
\end{example}
But there are also examples which do not fit into the scheme of Example~\ref{ex:ExtrinsicSplitting}: 

\begin{example}[The Segre embedding]\label{ex:Segre1}
Let positive integers $l,m,n$ with $n+1=(l+1)(m+1)$ be given and consider the parallel
isometric immersion $f:\C\rmP^l\times\C\rmP^m\to\C\rmP^n$\,,
$$
([z_0:\cdots :z_l],[w_0:\cdots :w_m])\mapsto [z_0w_0:z_0w_1:\cdots
:z_lw_m]\ \text{(all possible combinations)}\;,
$$
also known as the ``Segre embedding'' (cf.~\cite[p.~260]{BCO}). 
Note that $\C\rmP^n$ is an irreducible symmetric space.
\end{example}
We are especially interested in the question how the (purely intrinsic)
product structure of $M$ influences the extrinsic geometry of $f$\,.

\subsection*{This article is organized as follows:} 
The precise statement of our results and the required notation is given in Section~\ref{se:overview};
the main result of this article is Thm.~\ref{th:MainResult} from
Section~\ref{se:2.6}. The corresponding proofs are given in the subsequent
sections. The appendix provides some results on representations of Lie groups
and Lie algebras; in particular, there we discuss the isotropy representations of
symmetric spaces and their ``extrinsic tensor products''.

\section{Overview}
\label{se:overview}
Throughout this section, we always assume that $M$ and $N$ both are symmetric
spaces, that $M$ is additionally simply connected and without Euclidian factor
and that $f:M\to N$ is a parallel isometric immersion.
In the following, we implicitly identify $M$ with its de\,Rham
decomposition $M_1\times\cdots\times M_r$ (by means of some fixed isometry $M\to M_1\times\cdots\times M_r$).
Let $L_i\subset M$ be the canonical foliation (whose leafs 
are the various product slices $L_i(p):=\{p_1\}\times\cdots\times
M_i\times\cdots\times\{p_r\}$ through $p=(p_1,\cdots,p_r)\in M_1\times\cdots\times M_r$) and
$D^i:=TL_i$ for $i=1,\ldots,r$ be the corresponding
distribution of $M$\,.\footnote{Note that both $L_i$ and $D^i$ are in fact ``intrinsically defined'', i.e.\ they do not depend on
the chosen isometry $M\to M_1\times\cdots\times M_r$\,.}
Then all product slices of $M$ are simply connected, irreducible symmetric spaces. 

\subsection{The canonical decomposition of the first normal bundle}
\label{se:2.1}
We introduce the vector subbundles of $\bot f$ which are given by 
\begin{align}\label{eq:FirstNormalBundle}
  &\bbF:=\underset{p\in M}\bigcup\Spann{h(x,y)}{x,y\in T_pM}\;,\\
\label{eq:FirstNormalBundle1}
&\bbF^{ij}:=\underset{p\in M}\bigcup\Spann{h(x,y)}{(x,y)\in D^i_p\times
  D^j_p}\ \ \ \ \text{for}\ i,j= 1,\ldots,r\;,\\
\label{eq:FirstNormalBundle2}
&\tilde\bbF:=\sum_{i=1}^r\; \bbF^{ii}\;.
\end{align}
$\bbF$ is usually called the ``first normal bundle'' of $f$\,. 
Obviously, Eq.~\eqref{eq:FirstNormalBundle}-\eqref{eq:FirstNormalBundle2}
define parallel vector subbundles of $\bot f$\,; in particular, $\bbF$ is equipped with $\nabla^\bot$ (through restriction). 

Furthermore, let $\bbF^\sharp$ denote the maximal flat subbundle of
$\bbF$ and $\bbF^i_p\subset\bbF^{ii}_p$ denote the orthogonal complement of
$\bbF^\sharp_p\cap\bbF^{ii}_p$ in $\bbF^{ii}_p$ for each $p\in M$\,.

\begin{theorem}\label{th:decomposition}
\begin{enumerate}
\item The linear spaces $\bbF^{ij}_p$ and
$\bbF^{kl}_p$ are pairwise orthogonal for each $p\in M$ and $i,j=1,\ldots,r$ with  $i\neq j,\{i,j\}\neq
\{k,l\}$\,, and
\begin{align}\label{eq:FFF}
&\bbF=\tilde \bbF\oplus\bigoplus_{1\leq i<j\leq k}\bbF^{ij}\end{align}
is a fiberwise orthogonal decomposition into $\nabla^\bot$-parallel vector subbundles.
\item The linear spaces $\bbF^\sharp_p$ and
$\bbF^i_p$ are pairwise orthogonal for each $p\in M$ and $i=1,\ldots,r$\,.
The same is true for $\bbF^i_p$ and $\bbF^j_p$ with $i\neq j$\,.
Furthermore, $\bbF^\sharp_p\subset\tilde\bbF_p$ for each $p\in M$ and 
$$
\tilde\bbF=\bbF^\sharp\oplus\bigoplus_{i=1}^r\bbF^i
$$
is a fiberwise orthogonal decomposition into $\nabla^\bot$-parallel vector subbundles, too.
\end{enumerate}
\end{theorem}
A proof of this theorem is given in Section~\ref{se:3}.

\begin{corollary}\label{co:slice}
$f|L_i(p):L_i(p)\to N$ is a parallel isometric immersion for
each $p\in M$ and $i=1,\ldots,r$\,, too.
\end{corollary}

For a proof of this corollary see Section~\ref{se:3}.

\begin{example}[Continuation of Example~\ref{ex:Segre1}]\label{ex:Segre2}
Let $f:\C\rmP^l\times\C\rmP^m
\to\C\rmP^n$ be the ``Segre embedding''. Here all the product slices are totally geodesic submanifolds of $\C\rmP^n$\,,
hence $\tilde\bbF=0$ and thus $\bbF=\bbF^{1,2}$\,. Furthermore, $f$ is
``1-full'', i.e. $\bbF=\bot f$ and therefore $\dim(\bbF_p)=2\,m\,l$ for each $p\in \C\rmP^l\times\C\rmP^m$\,. 
\end{example}
\subsection{The $\scrA$-gradation on $\so(\osc_pf)$}
\label{se:2.2}
In the following, $T_pM$ is seen as a linear subspace of $T_{f(p)}N$ by
means of the injective map $T_pf:T_pM\to T_{f(p)}$ for each $p\in M$\,.
Then $V:=T_pM\oplus\bbF_p$ is also a subspace of $T_{f(p)}N$\,, usually 
called the (second) ``osculating space'' at $p$\,.
For each subspace $W\subset V$ we let $\sigma^W:V\to V$ denote 
the reflection in $W^\bot$\,, i.e. $\sigma^W|W=\Id$\,, $\sigma^W|W^\bot=-\Id$\,. Then the induced map $\Ad(\sigma^W):\so(V)\to\so(V),A\mapsto \sigma^W\circ A\circ\sigma^W$
is a linear involution on $\so(V)$\,. 
More specifically, let $i\in\{1,\ldots,r\}$ be given, consider the subspace of
$V$ which is given by
$$
D^i_p\oplus\bigoplus_{j\neq i}\bbF^{ij}_p
$$ 
and let $\sigma^i\in\rmO(V)$ denote the corresponding reflection. Then, according to Thm.~\ref{th:decomposition}, 
\begin{align}\label{eq:sigma1}
\forall x\in D^i_p:\;\sigma^i(x)&=-x\;,\\
\forall j\neq i\,,x\in D^j_p:\;\sigma^i(x)&=x\;,\\
\label{eq:sigma3}
\forall j\neq i\,,\xi\in \bbF^{ij}_p:\;\sigma^i(\xi)&=-\xi\;,\\
\forall \xi\in\tilde\bbF_p:\;\sigma^i(\xi)&=\xi\;,\\
\label{eq:sigma5}
\forall k\neq i,l\neq i\,,\xi\in\bbF^{kl}_p:\;\sigma^i(\xi)&=\xi\;.
\end{align}

Therefore, we have $\sigma^i(T_pM)=T_pM$\,, $\sigma^i(\bbF_p)=\bbF_p$ and
\begin{align}\label{eq:sigma2}
\forall x,y\in T_pM,i,j=1,\ldots,r:\;h(\sigma^i\,x,\sigma^i\,y)=\sigma^i\,h(x,y)\;.
\end{align}

Furthermore, $\sigma^i\circ\sigma^j|T_pM=\sigma^j\circ\sigma^i|T_pM$\,, hence, by the above, 
$\sigma^i\circ\sigma^j\,h(x,y)=\sigma^j\circ\sigma^i\,h(x,y)$ for
all $x,y\in T_pM$\,, thus $\sigma^i\circ
\sigma^j|\bbF_p=\sigma^j\circ\sigma^i|\bbF_p$ also holds, and therefore we even have
\begin{align}\label{eq:Vertauschen}
\forall i,j=1,\ldots,r:\;\sigma^i\circ \sigma^j=\sigma^j\circ\sigma^i\,.
\end{align}

\begin{definition}
Let $\scrA$ denote the Abelian group whose elements are the functions
$\delta:\{1,\ldots,r\}\to\{0,1\}$ and whose group structure is
given by $(\delta+\epsilon)\,(i):=\delta(i)+\epsilon(i)\ \mathrm{mod}\;2\,\Z$ for all
$\delta,\epsilon\in\scrA$ and $i=1,\ldots,r$\,.
Clearly, $$\scrA\cong\bigoplus_{i=1}^r\Z/2\,\Z\;.$$ 
\end{definition}

Let $\Ad:\rmO(V)\to\Gl(\so(V))$ denote the adjoint representation 
and $\Eig(\Ad(\sigma^i),\lambda)$ denote the corresponding
Eigenspace for each $\lambda\in\{-1,1\}$ and $i=1,\ldots,r$\,. By the above, 
$\{\Ad(\sigma_i)\}_{i=1,\ldots,r}$ is a commuting family of linear involutions on
$\so(V)$\,, hence $\so(V)_\delta:=\underset{i=1,\ldots,r}\bigcap
\Eig(\Ad(\sigma^i),(-1)^{\delta(i)})$ is a common Eigenspace of these involutions
for each $\delta\in \scrA$ and there is the splitting

\begin{equation}\label{eq:SplittingIntoEigenspaces}
\so(V)=\bigoplus_{\delta\in \scrA}\so(V)_{\delta}\;.
\end{equation}
Then

\begin{align}\label{eq:rules}
[\so(V)_{\delta},\so(V)_{\epsilon}]\subset\so(V)_{\delta+\epsilon}\ \text{for all}\ \delta,\epsilon\in \scrA\;.
\end{align}

In other words, $\so(V)$ carries the structure of an $\scrA$-graded Lie
algebra. If $A\in\so(V)_\delta$\,, then $A$ is called ``homogeneous'' and $\delta$ is called its ``degree''.

\begin{remark}\label{re:finer}
Let $\sigma:V\to V$ denote the reflection in the first normal space
$\bbF_p$ and $\so(V)_\pm$ denote the $\pm\,1$-eigenspace of
$\Ad(\sigma)$\,; hence $\so(V)=\so(V)_+\oplus\so(V)_-$\,. In this way,
$\so(V)$ may be also seen as a $\Z/2\,\Z$-graded Lie algebra.
As a consequence of~Eq.~\eqref{eq:sigma1}-\eqref{eq:sigma5}, \begin{equation}
\label{eq:factorization}\sigma=\sigma^1\circ\cdots\circ\sigma^r\;,
\end{equation}
thus 
\begin{equation}\label{eq:TotalDegree}
\scrA\to\{0,1\}\;, \delta\mapsto|\delta|:=\sum_{i=1}^r
\delta(i)\ \ \mathrm{mod}\;2\,\Z
\end{equation}
is a group homomorphism. Furthermore,
\begin{align}
    \label{eq:Even}
&\so(V)_+=\bigoplus_{|\delta|=0}\so(V)_\delta\;,\\
\label{eq:Odd}
&\so(V)_-=\bigoplus_{|\delta|=1}\so(V)_\delta\;.
\end{align}
 Hence, the $\scrA$-gradation is ``finer'' than the $\Z/2\,\Z$-gradation.
\end{remark}
\subsection{Curvature invariance of the linear spaces $\tilde\bbF_p$ and
  $\bbF^{ij}_p$ with $i\neq j$}
\label{se:2.3}
According to~\cite[Prop.~7]{J1}, for each $p\in M$ and all $\xi\,,\eta\in\bbF_p$ the
curvature endomorphism $R^N(\xi,\eta):T_{f(p)}N\to T_{f(p)}N, v\mapsto R^N(\xi,\eta)\,v
$
has the following properties:
\begin{align}\label{eq:CurvatureInvariance}
&R^N(\xi,\eta)(V)\subset V\;,\\
\label{eq:Fundamental_2}
&R^N(\xi,\eta)|V\in \so(V)_+\;.
\end{align} 
Here we additionally have:

\begin{theorem}\label{th:Fundamental}
If $\xi,\eta\in\bbF^{ij}_p$  or $\xi,\eta\in\tilde\bbF_p$\,, then, besides Eq.~\eqref{eq:CurvatureInvariance}, 
\begin{align}
\label{eq:Fundamental3}
&R^N(\xi,\eta)|V\in \so(V)_0\;.
\end{align}
\end{theorem}

\begin{corollary}\label{co:CurvatureInvariant}
$\tilde\bbF_p$ and $\bbF^{ij}_p$ are curvature invariant\footnote{\label{fn:Ci}A linear
  subspace $U\subset T_qN$ is called curvature invariant if $R^N(U\times U\times U)\subset U$ holds.}
subspaces of $T_{f(p)}N$ for each $p\in M$ and $i,j=1,\ldots,r$ with $i\neq j$\,.
\end{corollary}
A proof of Thm.~\ref{th:Fundamental} and Cor.~\ref{co:CurvatureInvariant} is given in Section~\ref{se:4}.

\subsection{Geometry of the second osculating bundle}
\label{se:2.4}
In the following, we
suppress  the injective vector bundle homomorphism $Tf:TM\to f^*TN$\,;
for convenience, the reader may simply assume that $M$ is a submanifold of $N$
and $f=\iota^M$\,. Then the ``pull-back'' $f^*TN$ is the vector bundle  over
$M$ which is given by $TM\oplus\bot f$\,.

\begin{definition}
\label{de:split-parallel}
\begin{enumerate}
\item 
The \emph{split connection} is 
the linear connection $\nabla^{\mathrm{sp}}:=\nabla^M\oplus\nabla^\bot$\ on $f^*TN$\,.
\item For each $p\in M$ let $\fetth:T_pM\to\so(T_{f(p)}N)$ be the linear map
defined by 
\begin{equation}\label{eq:fetth}
\forall\,x,y\in T_pM\,,\xi\in\bot_pM:\;\fetth(x)(y+\xi):=-S_\xi x+h(x,y)\;.
\end{equation}

\item Let $L$ be a second Riemannian space and $g:L\to M$ be any map. Sections
  of $f^*TN$ along $g$ which are parallel with respect to $\nabla^N$ or
  $\nabla^{\mathrm{sp}}$ are called ``$\nabla^N$-parallel'' and   ``split-parallel'', respectively.
\end{enumerate}
\end{definition}

Now the equations of Gau{\ss} and Weingarten can be formally combined to
\begin{align}
&\forall\,X\in\Gamma(TM),S\in\Gamma(f^*(TN)):\;
\label{eq:first_Gauss}
\Kovv{N}{X}{S}=\Kovv{\mathrm{sp}}{X}{S}+\fetth(X)\,S\;.
\end{align}
In the same way, the combined Equations of Gau{\ss}, Codazzi and Ricci for the
curvature are given by
\begin{equation}\label{eq:Gauss-Ricci}
\forall x,y\in T_pM:\;R^N(x,y)=R^{\mathrm{sp}}(x,y) +[\fetth(x),\fetth(y)]\;,
\end{equation}
where $R^{\mathrm{sp}}$ denotes the curvature tensor of $\nabla^{\mathrm{sp}}$\,.
Furthermore, the {\em second osculating bundle} 
\begin{equation}\label{eq:oscf}
 \osc f:=TM\oplus\bbF
\end{equation}
is a parallel vector subbundle of  $f^*TN$ with respect to
both $\nabla^N$ and $\nabla^{\mathrm{sp}}$ (see~\cite[Prop.~6]{J1}). In particular,
\begin{equation}\label{eq:ParallelSubbundle}
\forall p\in M, x,y\in T_pM:\;R^N(x,y)(\osc_pf)\subset\osc_pf\;.
\end{equation}

\begin{definition}\label{de:vector_bundle_maps}
\begin{enumerate}
\item
Let $\bbE$ be a vector bundle over $M$ and $\Hom(\bbE)$ denote the
linear space of {\em vector bundle maps} \/on $\bbE$\,.
More precisely, a map $F:\bbE\to\bbE$ belongs to
$\Hom(\bbE)$ if and only if there exists a
``base map'' $\underline{F}:M\to M$ such that
the following diagram is commutative
$$\begin{CD}
\bbE @>F>> \bbE\\
@VVV @VVV \\
M @>\underline{F}>> M
\end{CD}
$$
and $F|\bbE_p:\bbE_p\to\bbE_{\underline{F}(p)}$ is a linear map for each $p\in M$\,. 
Then we also say that $F$ is a vector bundle map along $\underline{F}$\,.
\item
An invertible vector bundle map is also called a {\em vector bundle
  isomorphism}. If $F\in\Hom(\bbE)$ even satisfies $F\circ F=\Id$\,, then $F$
will be called a {\em vector bundle involution}.

\item Now assume that $\bbE$ is equipped with a linear connection
  $\nabla^\bbE$\,. A vector bundle map $F:\bbE\to\bbE$ along $\underline{F}$ is called
  parallel, if $F\circ S$ is a parallel section along the curve
  $\underline{F}\circ c$ for every  curve $c:\R\to M$ and every parallel section $S:\R\to\bbE$ along $c$\,.
\end{enumerate}
\end{definition}

Equipping $\osc f$ with $\nabla^N$ or $\nabla^{\mathrm{sp}}$\,, we thus obtain the
notion of $\nabla^N$-parallel and split-parallel vector bundle maps on $\osc f$\,,
respectively. Then it is clear what is meant by a $\nabla^N$-parallel vector bundle
involution on $\osc f$\,.

For each 
$p=(p_1,\ldots,p_r)\in M_1\times\cdots\times M_r$ 
let $\sigma^i_p$ be the direct product map on $M_1\times\cdots\times M_r$ which is given by 
\begin{equation}\label{eq:Basemap}
\sigma^i_p:=\Id_{M_1}\times\cdots\times\Id_{M_{i-1}}\times
  \sigma^{M_i}_{p_i}\times\Id_{M_{i+1}}\times\cdots\times\Id_{M_r}\;,
\end{equation} 
where $\sigma^{M_i}_{p_i}$ denotes the corresponding geodesic symmetry for $i=1,\ldots,r$\,.

\begin{theorem}\label{th:Symmetries}
For each $p\in M$ there exists a  family $\{\Sigma_p^i\}_{i=1,\ldots,r}$ of pairwise
commuting, $\nabla^N$-parallel vector bundle involutions on $\osc f$ , characterized as follows:
\begin{itemize}
\item The base map of $\Sigma_p^i$ is given by Eq.~\eqref{eq:Basemap},
\item $\sigma_p^i(p)=p$ holds and $\Sigma_p^i|\osc_pf$ is the reflection $\sigma^i$
  described by Eq.~\eqref{eq:sigma1}-\eqref{eq:sigma5}.
\end{itemize}
\end{theorem}

A proof of this theorem is given in Section~\ref{se:5}.

\begin{remark}
Put $\Sigma_p:=\Sigma^1_p\circ\cdots\circ\Sigma^r_p$ for some $p\in M$\,. 
Then the base map of $\sigma_p$ is the geodesic symmetry of
$M$\,; furthermore, according to Eq.~\eqref{eq:factorization}, $\Sigma_p|\osc_pf$ is the reflection in $\bot^1_pf$\,.
Therefore, $\Sigma_p$ is the ``weak extrinsic symmetry'' at $p$ whose existence
was already proved in~\cite[Thm.~9]{J1}\,. Now we see how the intrinsic product structure of $M$ induces a distinguished
factorization of $\Sigma_p$\,.
\end{remark}

\subsection{The ``extrinsic holonomy Lie algebra''  of $\osc f$}
\label{se:2.5}
For each differentiable curve $c:[0,1]\to N$ let
$\ghdisp{0}{1}{c}{N}$ denote the parallel displacement in $TN$ along
$c$\,, i.e.
$$
\ghdisp{0}{1}{c}{N}\,S(0)=S(1)
$$
for all parallel sections $S:[0,1]\to TN$ along $c$ (cf. Def.~\ref{de:split-parallel}). 
In the following, we equip $\osc f$ with the linear connection $\nabla^N$ (see
Sec.~\ref{se:2.4})\,. Given a ``base point'' $p\in M$\,, we put
$V:=\osc_pf$\,; then 
\begin{align}
\label{eq:Hol(osc_M)}
&\Hol(\osc f):=\Menge{\ghdisp{0}{1}{f\circ c}{N}|V:V\to V}{c:[0,1]\to M\text{\ is a loop
    with }c(0)=p}\subset\rmO(V)\;.
\end{align}
is the holonomy group of $\osc f$ (see~\cite[Ch.~II and~III]{KN})\,.

\begin{definition}[\cite{J1}]\label{de:HolonomyLieAlgebra}
Let $\frakh$ denote the Lie algebra of $\Hol(\osc f)$\,; 
then $\hol(\osc f)$ is a subalgebra of $\so(V)$\,, called the ``extrinsic
holonomy Lie algebra of $\osc f$''.
\end{definition}

Non surprisingly, the geometric structure of $\osc f$ described by
Thm.~\ref{th:Symmetries} strongly influences the structure of $\frakh$\,.
In~\cite[Thm.~3]{J1}, we have already established the splitting
$\frakh=\frakh_+\oplus\frakh_-$ with $\frakh_\pm:=\frakh\cap\so(V)_\pm$\,.
Again, here we obtain a ``finer'' result (in the sense of Remark~\ref{re:finer}):

\begin{theorem}\label{th:hol}
\begin{enumerate}
\item $\frakh$ is an $\scrA$-graded subalgebra of $\so(V)$\,, i.e.\ there is the splitting
\begin{equation}\label{eq:splitting_of_hol}
\frakh=\bigoplus_{\delta\in \scrA}\frakh_{\delta}\;,\end{equation}
with $\frakh_\delta:=\frakh\cap\so(V)_\delta$ for each $\delta\in\scrA$\,.
\item Put $\frakh_i:=\frakh_{\delta_i}$ (see Lemma~\ref{le:delta-Funktionen}) and $\frakh^i:=\frakh_0\oplus
  \frakh_i$ for $i=1,\ldots,r$\,.
Then $\frakh^i$ is an $\scrA$-graded subalgebra of $\frakh$\,.\footnote{Actually, $\frakh^i$ is merely a $\Z/2\,\Z$-graded Lie algebra.} 
Furthermore, we have 
\begin{align}\label{eq:OuterDerivation1}
[\fetth(x),\frakh^i]\subset\frakh^i\;
\end{align}
for each $x\in D^i_p$ and $i=1,\ldots,r$\,. 
\item We have $R^N(x,y)\,V\subset V$ and
$R^N(\xi,\eta)\,V\subset V$ for all
$x,y\in D^j_p$ and $\xi,\eta\in\bbF^{kl}_p$ or $\xi,\eta\in\tilde\bbF_p$ ($k,l=1,\ldots,r$);
furthermore, 
\begin{align}
\frakh_0=\sum_{j=1}^r\Spann{R^N(x,y)|V}{x,y\in
  D^j_p}+\Spann{R^N(\xi,\eta)|V}{\xi,\eta\in\tilde\bbF_p}
+\sum_{k,l=1}^r\Spann{R^N(\xi,\eta)|V}{\xi,\eta\in\bbF^{kl}_p}\;.\label{eq:hol_plus}
\end{align}
\end{enumerate}
\end{theorem}

A proof of this theorem is given in Section~\ref{se:5}. 
\subsection{Homogeneity of parallel submanifolds}
\label{se:2.6}
Let $\Iso(N)$ denote the isometry group of $N$ (which is actually a Lie group, according to~\cite[Ch.~IV]{He}). 
Given a subset $M\subset N$\,, suppose that there exists a connected Lie subgroup $G\subset
\Iso(N)$ and some $p\in M$ such that $M$ is equal to the orbit $G\,p$\,.
Then $M$ is actually a submanifold\footnote{However, $M$ is not necessarily
  ``embedded'', i.e.\ its topology may be strictly finer than the ``subset
  topology'' (cf.~\cite[p.~7]{BCO})} of $N$\,, called a {\em homogeneous
  submanifold}. In this case, a standard argument shows that $M$ is even a complete Riemannian
manifold. Hence, if $f:M\to N$ is a parallel isometric immersion from a simply
connected symmetric space and $f(M)$ is a homogeneous submanifold of $N$\,,
then $f$ is necessarily a Riemannian covering onto $f(M)$\,. The following
stronger concept of extrinsic homogeneity was already used in~\cite{J2}:

\begin{definition}\label{de:homogeneous_holonomy}
Let $M$ be a submanifold of $N$\,.
We say that $M$ has {\em extrinsically homogeneous tangent holonomy
bundle}\/ if there exists a connected Lie subgroup $G\subset\Iso(N)$ with the
following properties:
\begin{itemize}
\item 
$g(M)=M$ for all $g\in G$\,.
\item  For each $p\in M$ and every curve $c:[0,1]\to M$ with $c(0)=p$ 
there exists some $g\in G$ such that $g(p)=c(1)$ and that the parallel displacement
along $c$ is given by
\begin{equation}\label{eq:f_is_equivariant}
\ghdisp{0}{1}{c}{M}=T_pg|T_pM:T_pM\to T_{c(1)}M\;.
\end{equation}
\end{itemize}
\end{definition}

\begin{definition}\label{de:curvature_isotropic}
\begin{enumerate}
\item
An (intrinsically) flat totally geodesic submanifold of $N$ is briefly called a
{\em flat} \/of $N$\,.  
\item According to~\cite[Ch.~V,~\S~6]{He}, the rank of $N$ is the maximal dimension of a flat of $N$\,.
\item According to~\cite[Ch.~V,~\S~1]{He}, $N$ is
called ``of compact type'' or ``of non-compact type'' if the 
Killing form of $\fraki(N)$ restricted to $\frakp$ is strictly negative
or strictly positive, respectively.\footnote{$N$ is of
compact (or non-compact) type if and only if the universal covering space of
$N$ is compact (or non-compact) (cf.~\cite[Ch.V-VII]{He}).}
\end{enumerate}
\end{definition}

Let $p$ be a fixed point of $M$\,; then $T_pM$ is seen as a subspace of
$T_{f(p)}N$\,, and so is $D^i_p$ for $i=1,\ldots,r$\,. Let $\exp^N$ denote the exponential spray
of $N$\,. Since $f|L_i(p)$ is also a parallel isometric immersion according
to Corollary~\ref{co:slice}, we can apply the Codazzi Equation to deduce that $D^i_p$ is even a curvature invariant subspace
of $T_{f(p)}N$ (cf. Fn.~\ref{fn:Ci}). Hence, by virtue of a result due to E.~Cartan,
\begin{equation}
\label{eq:tg}
\bar M_i:=\exp^N( D^i_p)\subset N
\end{equation} 
is a totally geodesically embedded symmetric space. Furthermore, let $\frakh$
denote the Lie algebra from Def.~\ref{de:HolonomyLieAlgebra}.

\begin{theorem}[Main Theorem]\label{th:MainResult}
Besides the conventions made at the beginning of Section~\ref{se:overview}, 
we also assume that $N$ is of compact or non-compact type and that
$\dim(M_i)\geq 3$ holds for $i=1,\ldots,r$\,.\footnote{\label{fn:1}
This condition simply means that $M$ does not split off a factor whose
dimension is 1 or 2 (cf. Def.~\ref{de:deRham})} Then the following
assertions are equivalent:
\begin{enumerate}
\item $f(L_i(p))$ is not contained in any flat of $N$ for $i=1,\ldots,r$\,.
\item $f(M)$ is a homogeneous submanifold.
\item $f(M)$ is a submanifold with extrinsically homogeneous
  tangent holonomy bundle. 
\item  We have \begin{equation}
\label{eq:fetth_in_hol}
\fetth(T_pM)\subset\frakh\,.
\end{equation}
\item The symmetric space $\bar M_i$ defined by Eq.~\eqref{eq:tg} is
  irreducible\footnote{\label{fn:irreducible} An arbitrary symmetric space is called ``irreducible'' if its universal covering space
    is an irreducible symmetric space in the sense of Def.~\ref{de:deRham}.} for $i=1,\ldots,r$\,.
\end{enumerate} 
\end{theorem}
The direction ``$(a)\Rightarrow (c)$'' should be seen as the main
result of this article; in case $M$ is even irreducible, this implication
already follows from~\cite[Thm.~5]{J2}. A proof of the Main Theorem can be found in Section~\ref{se:6}. 

The following is an immediate consequence of Theorem~\ref{th:MainResult}:

\begin{corollary} In the situation of Theorem~\ref{th:MainResult}, 
if all factors of $M$ are of dimension larger than the rank of
$N$\,, then $f(M)$ is a homogeneous submanifold of $N$\,.  
\end{corollary}

\section{Proof of Theorem~\ref{th:decomposition} and Corollary~\ref{co:slice}}
\label{se:3}
Let $M$ be a simply connected Riemannian submanifold without Euclidian factor
and $M_1\times\cdots\times M_r$ denote its de\,Rham decomposition (see Def.~\ref{de:deRham})\,.
Let $\Iso(M)^0$ and $\Iso(M_i)^0$ denote the connected components of the corresponding
isometry groups, respectively. We keep $p=(p_1,\ldots,p_r)\in M$ fixed and let $K$ and $K_i$ denote the
isotropy subgroups of $\Iso(M)^0$ and $\Iso(M_i)^0$ at $p$ and $p_i$\,,
respectively, for $i=1,\ldots,r$\,. Then, by means of the
uniqueness of the de\,Rham decomposition, we have 
\begin{equation}\label{eq:IsometryGroup}
\Iso(M)^0\cong \Iso(M_1)^0\times\cdots\times \Iso(M_r)^0\;,\end{equation} hence
\begin{equation}\label{eq:IsotropyGroup}
K\cong K_1\times\cdots\times K_r\;.
\end{equation}
Furthermore, $K$ and $K_i$ both are connected (since $M$ and its irreducible
factors are simply connected).

\begin{definition}\label{de:HomogeneousVectorbundle}
A \emph{homogeneous vector bundle} over $M$ is a pair $(\bbE,\alpha)$ 
where $\bbE\to M$ is a  vector bundle and
$\alpha:\Iso(M)^0\times\bbE\to\bbE$  is an action
through vector bundle isomorphisms (cf. Def.~\ref{de:vector_bundle_maps}~(b)) such that the
bundle projection of $\bbE$ is equivariant.
\end{definition}

In the situation of the last definition, 
\begin{equation}\label{eq:tau}
\Iso(M)^0\to M\,,\; g\mapsto g(p)
\end{equation}
is a principal fiber bundle with structure group $K$ and
$\bbE$ is a vector bundle associated therewith via
\begin{equation}\label{eq:assoziierung}
\Iso(M)^0\times \bbE_p\to\bbE\,,\;(g,v)\mapsto\alpha(g,v)\;.
\end{equation}
Therefore, one briefly writes ``$\bbE\cong
\Iso(M)^0\times_K\bbE_p$'' or ``$\bbE\cong
(\Iso(M)^0\times\bbE_p)/K$'' (cf.~\cite[Ch.~2.1]{J2}).

Now let $f:M\to N$ be a parallel isometric immersion and 
$\bbF$ denote the first normal bundle of $f$
(see Eq.~\eqref{eq:FirstNormalBundle}), equipped with the linear connection $\nabla^\bot$\,.

\begin{proposition}[{\cite[Prop.~10]{J1}}]\label{p:HomogeneousVectorbundle}
There exists an action $\alpha:\Iso(M)^0\times \bbF\to
\bbF$ through isometric, $\nabla^\bot$-parallel vector bundle isomorphisms 
characterized by
\begin{equation}\label{eq:alpha}
\alpha(g,h(x,y))=h(T_pg\,x,T_pg\,y)
\end{equation}
for all $x,y\in T_pM$ and $g\in\Iso(M)^0$\,. Hence $(\bbF,\alpha)$ is a homogeneous vector bundle over $M$\,.
\end{proposition}

Let $(\bbE,\alpha)$ be a homogeneous vector bundle over $M$ (compare
Def.~\ref{de:HomogeneousVectorbundle}) and $\fraki(M)$ denote the Lie algebra of $\Iso(M)$\,.
Following~\cite[Ch.~2.1]{J2}, 
the Cartan decomposition $\fraki(M)=\frakk\oplus\frakp$ induces a distinguished
connection $\nabla^\bbE$\,, called the
\emph{canonical connection}. It can be obtained as follows:
On the principal fiber bundle~\eqref{eq:tau}
there is a $\Iso(M)^0$-invariant connection $\scrH$ defined by
\begin{equation}\label{eq:scrH}
\scrH_g:=\Menge{X_g}{X\in\frakp}
\end{equation}
for all $g\in\Iso(M)^0$ where the elements of $\frakp$ are also considered as left-invariant
vector fields on $\Iso(M)^0$ (see~\cite[Vol.1,~p.~239]{KN}).
Since $\bbE$ is an associated vector bundle via Eq.~\eqref{eq:assoziierung},
the connection $\scrH$ induces a linear connection $\nabla^\bbE$\,. 
In order to relate  the parallel displacement induced by
$\nabla^\bbE$ to the horizontal structure $\scrH$\,, let a curve
$c:\R\to M$ with $c(0)=p$ be given; then
\begin{equation}\label{eq:pardisp}
\forall v\in \bbE_p:\;\ghdisp{0}{1}{c}{\nabla^\bbE}\,v=\alpha(\hat c(1),v)\;,
\end{equation}
where $\hat c:[0,1]\to \Iso(M)^0$ denotes the $\scrH$-lift of $c$ with $\hat
c(0)=\Id$\,. One can also show that the canonical connection does not depend on
the special choice of the base point $p$\,.

Let $\rho:K\to\rmO(T_pM)$ denote the isotropy representation, which equips $TM$
with the structure of a homogeneous vector bundle over $M$\,, too. Note that
$\rho$ is a faithful representation (since $K$ acts through isometries on $M$).
Furthermore, let $\Hol(M)\subset\rmO(T_pM)$ and
$\Hol(\bbF)\subset\rmO(\bbF_p)$ denote the corresponding holonomy groups,
respectively. Then $\Hol(M)$ and $\Hol(\bbF)$ both are connected
(since $M$ is simply connected).

\begin{lemma}\label{le:CanonicalConnection}
\begin{enumerate}
\item $\nabla^M$ is the canonical connection of $TM$\,.
\item $\rho$ maps $K$ isomorphically onto $\Hol(M)$ such that $K_i$
  corresponds to  $\Hol(M_i)$\,. Therefore, $K_i$ acts irreducibly and
  non-trivially on $D^i_p$ for $i=1,\ldots,r$\,. 
\item $\nabla^\bot$ is the canonical connection of $\bbF$ and the
  corresponding holonomy group $\Hol(\bbF)$ is the homomorphic image of $K$ which
  is given by \begin{equation}\label{eq:54321}
\Menge{\alpha(g,\,\cdot\,)|\bbF_p}{g\in K}\;.
\end{equation}
\end{enumerate}
\end{lemma} 
\begin{proof}
For~a): According to~\cite[Ch.~X.2]{KN}, 
the canonical connection $\nabla^{TM}$ is a metric and torsion-free
connection, hence $\nabla^{TM}$ is the Levi Civita connection $\nabla^M$\,. 

For~(b):
We have $\Hol(M)\cong\Hol(M_1)\times\cdots\times\Hol(M_r)$ (since $M$ the Riemannian
product of the $M_j$'s) and $K\cong
K_1\times\cdots\times K_r$ (see~\ref{eq:IsotropyGroup}). Thus
we can assume that $M$ is irreducible; then~\cite[Prop.~11]{J1} implies
that the first statement of~(b) holds on the level of Lie algebras. Thus also $\Hol(M)=\rho(K)$
holds, since the involved Lie groups are connected. 
Therefore, $K_i$ acts irreducibly  on $D^i_p$ as a consequence of the de\,Rham
Decomposition Theorem. In particular, $K_i$ acts non-trivially on $D^i_p$\,, since $\dim(M_i)\geq 2$\,.

For~c):
Let a point $p\in M$\,, a curve $c:[0,1]\to M$ with $c(0)=p$ and some $g\in\Iso(M)^0$ with $g(p)=p$ be
given\,. Let $\hat c$ denote the horizontal lift of $c$ with $\hat
c(0)=\Id$\,. Then it suffices to show that the parallel displacement
described in Eq.~\eqref{eq:pardisp} (with $\bbE:=\bbF$) is equal to the parallel
displacement with respect to $\nabla^\bot$\,. 

For this: Let $\xi\in\bbF_p$ be given; thereby, we will assume 
that $\xi=h(x,y)$ for certain $x,y\in T_pM$\,.
Then, in accordance with Part~(a) and Eq.~\eqref{eq:pardisp}, 
the parallel displacement of $x$ and $y$ in $TM$ along $c$ is given by $T\hat
c(1)\,x$ and $T\hat c(1)\,y$\,, respectively. 
Then $$\alpha(\hat c(1),\xi)\stackrel{\eqref{eq:alpha}}{=}h(T\hat c(1)\,x,T\hat c(1)\,y)$$
is the parallel displacement of $\xi$ along $c$ with respect to
$\nabla^\bot$ (since $h$ is a parallel section of $\rmL^2(TM,\bot f)$). 

The previous arguments also show that $\Hol(\bbF)$ is given by
Eq.~\eqref{eq:54321}.
\end{proof} 

The proof of the following lemma is ``straight forward'':

\begin{lemma}\label{le:HomogeneousVectorBundle1}
Let $(\bbE,\alpha)$ be a homogeneous vector bundle over $M$ and $\nabla^\bbE$
denote the corresponding canonical connection. Let $p\in M$ be a fixed point and $K$ be the corresponding 
isotropy subgroup of $\Iso(M)^0$\,.
Assigning to each
$\nabla^\bbE$-parallel section $s$ its value $s(p)\in\bbE_p$ induces a one-one
correspondence 
\begin{equation}\label{eq:correspondence}
\text{parallel sections of } \bbE\ \leftrightarrow\ K\text{-invariant elements
  of }\bbE_p\,.
 \end{equation}
\end{lemma}

\paragraph{Proof of Thm.~\ref{th:decomposition}}
\begin{proof}
For~(a): Clearly,
\begin{equation}\label{eq:FirstNormalBundle3}
\bbF=\tilde\bbF+\sum_{1\leq i<j\leq r}\bbF^{ij}\;.
\end{equation}
Now let $i,j,k,l$ with $i\neq j$ and $\{i,j\}\neq \{k,l\}$ be given; 
hence, it remains to show that $\bbF^{ij}_p\bot\bbF^{kl}_p$ holds.

For this: Without loss
of generality, we will assume that $i$ is different from both $k$ and
$l$\,. Furthermore, we note that $K_i$ is a Lie subgroup of $K$ according to 
Eq.~\eqref{eq:IsotropyGroup}. Hence, on the one hand, we have the action of
$K_i$ on $\bbF_p$ through orthogonal transformations (by means of $\alpha$);
then $\bbF^{kl}_p$ and $\bbF^{ij}_p$ both are invariant subspaces of $\bbF_p$\,.

On the other hand, $K_i$ acts orthogonally on $T_pM$ by means of the isotropy
representation. Then $D^i_p$ is an irreducible invariant subspace of $T_pM$ (according to
Lemma~\ref{le:CanonicalConnection}~(b)) and $D^j_p$ is an invariant subspace on which $K_i$ acts trivially (since
$i\neq j$). Therefore, the induced action of $K_i$ on $D^i_p\otimes D^j_p$ 
(as described in Def.~\ref{de:ExteriorTensorprodukt2}) is isomorphic to the direct sum of $m_j$ copies of
$D^i_p$ (where $m_j$ denotes the dimension of $M_j$), whereby the action of
$K_i$ on each copy of $D^i_p$ is irreducible and non-trivial.

Therefore, and since $h:D^i_p\otimes D^j_p\to\bbF^{ij}_p$ is a surjective
homomorphism in accordance with Eq.~\eqref{eq:alpha}, 
$\bbF^{ij}_p$ also decomposes into a direct sum of non-trivial and irreducible subspaces,
by virtue of Lemma~\ref{le:Schur2}~(d) (here and in the following, 
we use the ``Lie group version'' of this Lemma, cf. Rem.~\ref{re:Reformulate}). 
On the other hand, since $i$ is different from both $k$ and
$l$\,, $K_i$ acts trivially on $D^k_p\otimes D^l_p$\,. Thus
$K_i$ also acts trivially on $\bbF^{kl}_p$\,, because of Eq.~\eqref{eq:alpha}.
Thus $\bbF^{ij}_p\bot\bbF^{kl}_p$\,, according to Lemma~\ref{le:Schur2}~(a)\,.

For~(b): It suffices to show that
\begin{align}\label{eq:1111}
&\bbF^\sharp_p\subset\tilde\bbF_p\;,\\
\label{eq:1112}
&\bbF^\sharp_p\bot\bbF^i_p\ \text{for}\ i=1,\ldots,r\;,\\
\label{eq:1113}
&\bbF^i_p\bot\bbF^j_p\ \text{for}\ i\neq j\;.
\end{align}

For this: Since $M$ is simply connected, $\bbF^\sharp$ is pointwise spanned
by the parallel sections of $\bbF$\,. Hence sections of
$\bbF^\sharp$ correspond uniquely to $K$-invariant elements of $\bbF_p$\,, 
according to Lemma~\ref{le:CanonicalConnection}~(a) in combination with Lemma~\ref{le:HomogeneousVectorBundle1}.
Therefore, $\bbF^\sharp_p$ is the maximal subspace of $\bbF_p$ on which $K$
acts trivially. Since we have already seen in the proof of Part~(a) 
that $\bbF^{ij}_p$ is isomorphic to a direct sum of non-trivial and irreducible
$K_i$-modules, it follows by arguments given before that
$\bbF^\sharp_p\bot\bbF^{ij}_p$ for all $i\neq j$\,; now Part~(a) implies that Eq.~\eqref{eq:1111} holds.

Let $\bbF^{ii}_p=W_0\oplus
W_1\oplus\cdots\oplus W_l$ be a decomposition into invariant subspaces such that $W_0$ is a trivial
$K_i$-module and that $W_k$ is an irreducible and non-trivial $K_i$-module
for $k=1,\ldots,l$\,. Then $W_0=\bbF^\sharp_p\cap\bbF^{ii}_p$ holds, since
$K_j$ acts trivially on $\bbF^{ii}_p$ for $j\neq i$ anyway, according to Eq.~\eqref{eq:alpha}.
Therefore, $\bbF^i_p$ is given by $W_1\oplus\cdots\oplus W_l$\,. From
Lemma~\ref{le:Schur2}~(a) we now conclude that Eq.~\eqref{eq:1112} is valid. 
Finally, since $K_i$ acts trivially on $\bbF^j_p$ for $i\neq
j$\,, we obtain from a similar argument that Eq.~\eqref{eq:1113} also holds\,.

\end{proof}

\begin{lemma}\label{le:hilf1}
We have  
\begin{align}\label{eq:hilf1}
&\forall x_i\in D^i_p,x_j\in
D^j_p,\xi\in\bbF^{ij}_p:\,S_\xi\,x_i\in  D^j_p\;,\\
\label{eq:hilf2}
&\forall x_j\in D^j_p,\xi\in\bbF^{ii}_p:\;S_\xi\,x_j\in D^j_p\;,\\
\label{eq:hilf3}
&\forall x_i,y_i\in D^i_p:\;R^N(x_i,y_i)(D^j_p\oplus\bbF^{ij}_p)\subset D^j_p\oplus\bbF^{ij}_p
\end{align}
for all $i,j=1,\ldots,r$\,.
\end{lemma}
\begin{proof}
For Eq.~\eqref{eq:hilf1}: For each $x_k\in D^k_p$ with $k\neq j$ we have
$$\g{S_\xi\,x_i}{x_k}=-\g{h(x_i,x_k)}{\xi}=0\;,$$
according to Thm.~\ref{th:decomposition}. Hence
$S_\xi\,x_i\in D^j_p$\,.

For Eq.~\eqref{eq:hilf2}: For each $x_k\in D^k_p$ with $k\neq j$ we have
$$\g{S_\xi\,x_j}{x_k}=-\g{h(x_j,x_k)}{\xi}=0\;,$$
by means of Thm.~\ref{th:decomposition}. Hence $S_\xi\,x_j\in
D^j_p$\,.

For Eq.~\eqref{eq:hilf3}: We have $[\fetth(x_i),\fetth(y_i)]\,\bbF^{ij}_p\subset
\bbF^{ij}_p$ and $[\fetth(x_i),\fetth(y_i)]\,D^j_p\subset
D^j_p$ for all $x_i,y_i\in D^i_p$\,, according
to Eq.~\eqref{eq:hilf1}. Moreover, $R^M(x_i,y_i)\,D^j_p\subset
D^j_p$ (since $D^j$ is a parallel vector subbundle of $TM$) 
and $R^\bot(x_i,y_i)\,\bbF^{ij}_p\subset
\bbF^{ij}_p$ (since $\bbF^{ij}$ is parallel vector subbundle of $\bbF$). Now the result follows from Eq.~\eqref{eq:Gauss-Ricci}.
\end{proof}

\paragraph{Proof of Cor.~\ref{co:slice}}
\begin{proof}
Put $f_i:=f|L_i(p):L_i(p)\to N$\,; then 
$$
T_qL_i(p)=D^i_q\ \text{and}\ \bot_qf_i=\bigoplus_{j\neq i}D^j_q\oplus\bot_qf\;,
$$
holds for each $q\in M$\,. Furthermore, since $L_i(p)\subset M$ is totally geodesic, the second fundamental form of $f_i$
is given by $$h|D^i_q\times D^i_q$$ for each $q\in L_i(p)$\,; hence the first normal bundle of $f_i$ is given by
$\bbF^{ii}|L_i(p)$ (pullback of $\bbF$ to $L_i(p))$)\,. 
I claim that $\nabla^\bot$ coincides on $\bbF^{ii}|L_i(p)$ with the usual normal
connection of $f_i$\,.

For this: Let $\xi$ be a section of $\bbF^{ii}$ along $L_i(p)$\,. Then we have $\nabla^N_X\xi=\nabla^\bot_X\xi-S_\xi(X)$\,,
where $\nabla^\bot_X\xi$ again is a section of $\bbF^{ii}$ along $L_i(p)$ (because
$\bbF^{ii}$ is a parallel vector subbundle of $\bbF$) and
$S_\xi(X)$ again is a section of $D^i$ along $L_i(p)$\,, according to
Lemma~\ref{le:hilf1}~(b). Therefore, by means of the Weingarten equation, the
covariant derivative of $\xi$ with respect to the normal connection of $\bot f_i$
is given by $\nabla^\bot_X\xi$\,.
 
Now let $c:\R\to L_i(p)$ be a curve and $X,Y$ be parallel sections of $TL_i(p)$ along
$c$\,. Then $c$ is a curve into $M$ and $X,Y$ are parallel sections of $TM$ along
$c$\,, too. Hence $t\mapsto\xi(t):=h(X(t),Y(t))$ defines a parallel section of
$\bbF^{ii}$ along $c$\,, and hence, by the previous, $\xi$ is also a parallel
section of $\bot f_i$ with respect to the usual normal connection of $f_i$\,. Therefore, $f_i$ is a parallel isometric immersion. 
\end{proof}

\section{Proof of Theorem~\ref{th:Fundamental} and Corollary~\ref{co:CurvatureInvariant}}
\label{se:4}
We continue with the notation from Section~\ref{se:2.2}. For $i\in \{1,\ldots,r\}$ let $\delta_i$ denote
the characteristic function of $\{i\}$\,, i.e. 
\begin{equation}\label{eq:CharacterisiticFunction}
\delta_i(i)=1\ \text{and}\ \delta_i(j)=0\ \text{for}\ j\neq i\;.
\end{equation}

\begin{lemma}\label{le:delta-Funktionen}
\begin{enumerate}
\item
Let $0$ denote the zero-function on $\{1,\ldots,r\}$\,. Then
\begin{align}\label{eq:unit}
\so(V)_0=\Menge{&A\in\so(V)_+}{A(D^i_p)\subset D^i_p\ \text{and}\ 
A(\bbF^{ij}_p)\subset\bbF^{ij}_p\ \text{for all}\ i,j=1,\ldots,r\ \text{with}\ j\neq i\;.}
\end{align}

\item Put $\so(V)_i:=\so(V)_{\delta_i}$ for $i=1,\ldots,r$\,. Then we have 
\begin{align}\label{eq:Degree0}
\so(V)_i=\Menge{A\in\so(V)_-}{A(D^i_p)\subset\tilde\bbF_p\ \text{and}\ 
A(D^j_p)\subset\bbF^{ij}(p)\ \text{for all}\ j\neq i\;.}
\end{align}

Therefore,
\begin{equation}\label{eq:Degree1}
\fetth(x)\in\so(V)_i\ \text{for all}\ x\in D^i_p\ \text{and}\ i=1,\ldots,r\;.
\end{equation}
\item  Put
  $\so(V)_{ij}:=\so(V)_{\delta_i+\delta_j}$ for all $i,j=1,\ldots,r$\,. 
Then, in addition to Eq.~\eqref{eq:ParallelSubbundle}, we have 
\begin{align}
\label{eq:Degree2}
\forall (x,y)\in D^i_p\times D^j_p:\;R^N(x,y)|V\in\so(V)_{ij}\;.
\end{align}
In particular (since $\delta_i+\delta_i=0$ holds), 
\begin{align}\label{eq:Degree3}
\forall x,y\in D^i_p:\;R^N(x,y)|V\in\so(V)_0\;.
\end{align}
\end{enumerate}
\end{lemma}
\begin{proof}
Using Eq.~\eqref{eq:sigma1}-\eqref{eq:sigma5}, the proof
of Eq.~\eqref{eq:unit} and~\eqref{eq:Degree0} is straightforward. 
Then Eq.~\eqref{eq:Degree1} immediately follows from Eq.~\eqref{eq:fetth} and~\eqref{eq:Degree0}. 

For Eq.~\eqref{eq:Degree2} and~\eqref{eq:Degree3}: Note that the curvature tensor
for the split-connection of $f^*TN$ (see Def.~\ref{de:split-parallel}) is given by 
$R^{\mathrm{sp}}(x,y)=R^M(x,y)\oplus R^\bot(x,y)$ for all $x,y\in T_pM$\,,
hence $R^{\mathrm{sp}}(x,y)(V)\subset V\ \text{and}\ R^{\mathrm{sp}}(x,y)|V\in
\so(V)_0$ for all $x,y\in T_pM$ as a consequence of Eq.~\eqref{eq:unit} and since Eq.~\eqref{eq:FirstNormalBundle}-\eqref{eq:FirstNormalBundle2} 
are parallel vector subbundles of $\bot f$\,. Therefore, and by means of
Eq.~\eqref{eq:Gauss-Ricci} and~\eqref{eq:Degree1}, Eq.~\eqref{eq:Degree3}
now follows. Moreover, in case $i\neq j$\,,
$R^M(x,y)=0$ for all $(x,y)\in D^i_p\times D^j_p$\,, hence, according to~\cite[Prop.~4~(d)]{J1}, 
$$
R^\bot(x_1,y_1)h(x_2,y_2)=h(R^M(x_1,y_1)\,x_2,y_2)+h(x_2,R^M(x_1,y_1)\,y_2)=0
$$
for all $(x_1,y_1)\in D^i_p\times D^j_p$\,, $(x_2,y_2)\in T_pM\times
T_pM$\,, i.e. $R^\bot(x_1,y_1)|\bbF_p=0$\,. Consequently,
$R^N(x,y)|V=[\fetth(x),\fetth(y)]$ for all $(x,y)\in D^i_p\times
 D^j_p$\,, according to Eq.~\eqref{eq:Gauss-Ricci}; therefore, and by means of arguments
given before, now Eq.~\eqref{eq:Degree2} also follows for $i\neq j$\,.
\end{proof}

\paragraph{Proof of Thm.~\ref{th:Fundamental}}\label{se:4.1}
\begin{proof}
According to~\cite[Lemma~5]{J1} (see in particular Ed.~(60) there), for each
$p\in M$ the
following Equation holds on $V$ for all $x_1,\ldots x_4\in T_pM$\,:
\begin{align}\notag
[\fetth(x_1),[\fetth(x_2),R^N(x_3,x_4)]]=&R^N(\fetth(x_1)\fetth(x_2)\,x_3,x_4)+R^N(x_3,\fetth(x_1)\fetth(x_2)\,x_4)\\
&+R^N(\fetth(x_1)\,x_3,\fetth(x_2)\,x_4)+R^N(\fetth(x_2)\,x_3,\fetth(x_1)\,x_4)\;.
\label{eq:ZweiteAbleitung}
\end{align}

Furthermore, Eq.~\eqref{eq:fetth},\eqref{eq:hilf1} and~\eqref{eq:hilf2} imply that 
\begin{align}\label{eq:SimpleRelation1}&\fetth(D^i_p)\,\fetth(D^j_p)(D^i_p)=\fetth(D^i_p)\,\fetth(D^i(0))(D^j_p)
\subset D^j_p\;,\\
&\fetth(D^i_p)\,\fetth(D^j_p)(D^j_p)\subset D^i_p
\label{eq:SimpleRelation2}
\end{align}  
for all $i,j=1,\ldots,r$\,.
Using that $R^N(h(x_j,x_i),h(x_i,x_j))=0$ (because of the symmetry of $h_p$)\,, Eq.~\eqref{eq:ZweiteAbleitung}
implies that for all $x_i,y_i\in D^i_p$ and $x_j,y_j\in D^j_p$
the following equations hold on $V$\,,
\begin{align}
R^N(h(x_i,x_i),h(y_j,y_j))
\label{eq:Fundamental0}
=&[\fetth(x_i),[\fetth(y_j),R^N(x_i,y_j)]]
-R^N(\fetth(x_i)\,\fetth(y_j)\,x_i,y_j)-R^N(x_i,\fetth(x_i)\,\fetth(y_j)\,y_j)\;,\\
\notag
R^N(h(x_i,x_j),h(y_i,y_j))
=&[\fetth(x_i),[\fetth(y_j),R^N(x_j,y_i)]]-R^N(\fetth(x_i)\,\fetth(y_j)\,x_j,y_i)-R^N(x_j,\fetth(x_i)\,\fetth(x_j)\,y_i)\\&-R^N(h(x_i,y_i),h(x_j,y_j))\;.\label{eq:Fundamental1}
\end{align}
In order to establish Eq.~\eqref{eq:Fundamental3}, first assume that $\xi,\eta\in\tilde\bbF_p$\,. Furthermore, without loss of generality we
can assume that there exists $(x,y)\in D^i_p\times D^j_p$ such that
$\xi=h(x,x)$\,, $\eta=h(y,y)$ (since $h_p$ is symmetric).
Then Eq.~\eqref{eq:Fundamental3} is an immediate consequence
of Eq.~\eqref{eq:rules},\eqref{eq:Degree1},\eqref{eq:Degree2},\eqref{eq:SimpleRelation1},\eqref{eq:SimpleRelation2} and~\eqref{eq:Fundamental0}.

Now assume that $\xi,\eta\in\bbF^{ij}_p$\,. 
Using Eq.~\eqref{eq:Fundamental1} in combination with the
previous, Eq.~\eqref{eq:Fundamental3} now follows by similar arguments as above.
\end{proof}

\paragraph{Proof of Cor.~\ref{co:CurvatureInvariant}}
\begin{proof}
Use Eq.~\eqref{eq:unit} in combination with Thm.~\ref{th:Fundamental}.
\end{proof}

\section{Proof of Theorem~\ref{th:Symmetries} and Theorem~\ref{th:hol}}
\label{se:5}
The proof of the following basic lemma is left to the reader:

\begin{lemma}[Continuation of Lemma~\ref{le:HomogeneousVectorBundle1}]\label{le:HomogeneousVectorBundle2}
Let $(\bbE,\alpha)$ be a homogeneous vector bundle over $M$ and $\nabla^\bbE$
denote the corresponding canonical connection. Let $p\in M$ be a fixed point and $K$ be the corresponding 
isotropy subgroup of $\Iso(M)^0$\,. 

\begin{enumerate}
\item For each $\sigma\in\Iso(M)$\,, the ``pull back bundle'' $\sigma^*\bbE\to M$ (whose total space is the fiber product
$M\times_\sigma\bbE$) is a homogeneous vector bundle over $M$\,, too, 
by means of the action $$\Iso(M)^0\times\sigma^*\bbE\to\sigma^*\bbE,(g,v)\mapsto
\alpha(\sigma\circ g\circ\sigma^{-1},v)\;.$$
Moreover, the corresponding canonical connection is the one which is induced
by $\nabla^\bbE$ in the usual way.
\item If $(\tilde \bbE,\tilde\alpha)$ is a second homogeneous vector bundle over $M$\,,
  then the induced vector bundle $\rmL^1(\bbE,\tilde\bbE)$ 
is a homogeneous vector bundle over $M$\,, too, by  means of the action 
$$\Iso(M)^0\times\rmL^1(\bbE,\tilde\bbE)\to\rmL^1(\bbE,\tilde\bbE),(g,\ell:\bbE_p\to\tilde\bbE_p)\mapsto
\tilde\alpha_g|\tilde\bbE_p\circ\ell\circ\alpha_g^{-1}|\bbE_{g(p)}\;.$$
Moreover, if $\nabla^{\tilde\bbE}$ denotes the canonical connection of $\tilde\bbE$\,, then the canonical connection on $\rmL^1(\bbE,\tilde\bbE)$ 
is the one which is induced by $\nabla^\bbE$ and $\nabla^{\tilde\bbE}$ in the usual way.
\end{enumerate}
\end{lemma}

The proof of the following lemma is also straightforward:

\begin{lemma}\label{le:splitting}
Let $V'$ be a Euclidian vector space and $W\subset V'$ be a subspace. Let
$\sigma\:V'\to V'$ denote the reflection in $W^\bot$ and
$\so(V')=\so(V')_+\oplus\so(V')_-$ be the decomposition into 
the eigenspaces of $\Ad(\sigma)$ (see Remark~\ref{re:finer}). Then the natural map
$\so(V')_-\to\rmL(W,W^\bot),A\mapsto A|W$ is a linear isomorphism.
\end{lemma}

\paragraph{Proof of Thm.~\ref{th:Symmetries}}
\begin{proof}
\textbf{1. Step:} 
Let $p$ be a fixed point and let $\sigma^i_p$ be defined according to Eq.~\eqref{eq:Basemap}; then $\sigma_p^i(p)=p$ holds.
We will show that $\bbF$ admits a parallel vector bundle
isomorphism $I_p^i$ along the base map $\sigma_p^i$ such that 
\begin{align}
&I_p^i|\bbF_p=\sigma^i|\bbF_p\;,\\
\label{eq:Compatible}
&\forall q\in M,\,x,y\in T_qM:\;I_p^ih(x,y)=h(T\sigma_p^i\,x,T\sigma_p^i\,y)
\end{align}
holds for $i=1,\ldots,r$\,.

For this: As a consequence of Lemma~\ref{le:CanonicalConnection}~(c),
Lemma~\ref{le:HomogeneousVectorBundle1} and Lemma~\ref{le:HomogeneousVectorBundle2}, 
every parallel vector bundle homomorphism of $\bbF$ along $\sigma^i_p$ 
uniquely corresponds to a parallel section of
$\bbE:=\rmL^1(\bbF,\sigma_p^{i*}\bbF)$\,, where the latter is seen as a homogeneous
vector bundle equipped with the corresponding canonical connection as described
in Lemma~\ref{le:HomogeneousVectorBundle2}.
We let $K$ be the isotropy subgroup of
$\Iso(M)$ at $p$\,; then $\bbF^{ij}_p$ and $\bbF^{ij}_p$ are subspaces
of $\bbF_p$ which are invariant under the action of $K$\,; hence the 
linear map $\sigma^i|\bbF_p$ (defined by Eq.~\eqref{eq:sigma3}-\eqref{eq:sigma5})
is a $K$-invariant element of $\bbE_p$\,. 
Therefore, according to Eq.~\eqref{eq:correspondence}, $\sigma^i|\bbF_p$ uniquely extends
to a parallel section $I_p^i$ of $\bbE$\,, again by means of
Lemma~\ref{le:HomogeneousVectorBundle1}. Then the base map $\sigma^i_p$ is an involution on
$M$ and $I_p^i|\bbF_p=\sigma^i|\bbF_p$ is a reflection of
$\bbF_p$\,, hence the parallelity of $I_p^i$ implies that $I_p^i|\bbF_p$ is a
vector bundle involution. It remains to
establish Eq.~\eqref{eq:Compatible}:

Let $c:[0,1]\to M$ be a curve
with $c(0)=p$ and $c(1)=q$ and
let $X,Y$ be parallel sections of $TM$ along $c$ with $X(1)=x$
and $Y(1)=y$\,. Consider the two sections $S_1$ and $S_2$ of $\bot^1f$ along the curve
$\sigma_p^i\circ c$ defined by $S_1(t):=I_p^i(h(X(t),Y(t))$ and $S_2(t):=h(T\sigma_p^i\,X(t),T\sigma_p^i\,Y(t))$\,.
Using the parallelity of $f$ and the fact that $\sigma_p^i$ is an isometry of $M$\,,
we see that both $S_1$ and $S_2$ are parallel sections. Furthermore
$S_1(0)=S_2(0)$ holds, in accordance with Eq.~\eqref{eq:sigma2};
therefore $S_1=S_2$\,, in particular Eq.~\eqref{eq:Compatible} holds.

\textbf{2. Step:}
Put $\Sigma_p^i:=T\sigma_p^i\oplus I_p^i$\,. Then $\Sigma_p^i$
  is a split-parallel vector bundle involution along $\sigma_p^i$\,, and
  $\Sigma_p^i|V$ is the reflection $\sigma^i$
  described by Eq.~\eqref{eq:sigma1}-\eqref{eq:sigma5}. I claim that $\Sigma_p^i$
  is also $\nabla^N$-parallel:

\eqref{eq:Compatible} in combination with Lemma~\ref{le:splitting} implies that
\begin{equation}
\label{eq:Compatible_2}
\forall q\in M,x\in T_qM,v\in\osc_qf:\;\Sigma_p^i(\fetth(x)\,v)=\fetth(T\sigma_p^i\,x)(\Sigma_p^i\,v)\;.
\end{equation}
Since $\Sigma_p^i$ is split-parallel, Eq.~\eqref{eq:Compatible_2} combined with the
Gau{\ss}-Weingarten equation Eq.~\eqref{eq:first_Gauss} 
implies that $\Sigma_p^i$ is $\nabla^N$-parallel, too. The result follows.

The last assertion of the theorem follows from Eq.~\eqref{eq:Vertauschen} in
combination with the parallelity of $\Sigma_p^i$\,. 
\end{proof}

\paragraph{Proof of Thm.~\ref{th:hol}}
\begin{proof}
For~(a): Let $p$ be a fixed point, put $V:=\osc_pf$ and let $\sigma^i$ denote the reflection of
$V$ defined by Eq.~\eqref{eq:sigma1}-\eqref{eq:sigma5}.
Then it suffices to show that 
\begin{equation}\label{eq:hol_2}
\Ad(\sigma^i)(\frakh)=\frakh\;.
\end{equation}
Let $\Sigma_p^i$ denote the symmetry of $\osc f$ described in Thm.~\ref{th:Symmetries}, and let $c:[0,1]\to M$ be a loop with $c(0)=p$\,.
Remember that $\Sigma_p^i$ is a $\nabla^N$-parallel vector bundle isomorphism of $\osc f$ along $\sigma_p^i$ (see Eq.~\eqref{eq:Basemap}) with
$\Sigma_p^i|\osc_pf=\sigma^i$\,, in accordance with Thm.~\ref{th:Symmetries}; hence
\begin{equation}
\sigma^i\circ\ghdisp{0}{1}{c}{N}|V=\ghdisp{0}{1}{\sigma_p^i\circ c}{N}\circ \sigma^i\;.
\end{equation}
From the last line in combination with Eq.~\eqref{eq:Hol(osc_M)} we
conclude that $\Hol(\osc f)$ is invariant under group conjugation with
$\sigma^i$\,; thus Eq.~\eqref{eq:hol_2} holds.

For~(b): In accordance with~\cite[Thm.~3]{J1} 
\begin{equation}\label{eq:OuterDerivation2}
[\fetth(x),\frakh]\subset\frakh
\end{equation}
holds for each $x\in T_pM$\,. Now Eq.~\eqref{eq:OuterDerivation1} follows as an immediate
consequence of Part~(a) in combination with Eq.~\eqref{eq:Degree1} and the rules for graded Lie algebras (see Eq.~\eqref{eq:rules}).

For~(c): First, let us see that r.h.s.\ of~\eqref{eq:hol_plus} is contained in
$\frakh_0$\,.

For this: We have $R^N(x,y)(V)\subset V$ and $R^N(x,y)|V\in\frakh_0$ for all $x,y\in
D^j_p$ and $j=1,\ldots,r$\,, 
because of Eq.~\eqref{eq:Degree3} and the Theorem of
Ambrose/Singer (cf.~\cite[Proof of Thm.~3]{J1}). 
Now let $\xi\in \bbF^{jj}_p$\,, $\eta\in\bbF^{l,l}_p$ be given; thereby, we will 
assume that there exist $(x_j,x_k)\in D^j_p\times D^k_p$ such that $\xi=h(x_j,x_j)$ and $\eta=h(x_k,x_k)$\,.
Then, in accordance with Eq.~\eqref{eq:CurvatureInvariance} and~\eqref{eq:Fundamental3}
we have $R^N(\xi,\eta)(V)\subset V$ and $R^N(\xi,\eta)|V\in\so(V)_0$\,;
moreover, as a consequence of Eq.~\eqref{eq:Fundamental0} combined
with Eq.~\eqref{eq:OuterDerivation2}, even $R^N(\xi,\eta)|V\in\frakh_0$ holds.
Additionally using Eq.~\eqref{eq:Fundamental1}, we obtain the same 
result for all $\xi,\eta\in \bbF^{jl}_p$ and $j,l=1,\ldots,r$\,, which proves
our claim.

In order to finally establish Eq.~\eqref{eq:hol_plus}, we introduce the following linear subspaces of $\so(V)$\,,
\begin{align*}
&\frakj_0:=\Spann{R^N(y_1,y_2)|V}{y_1,y_2\in T_pM}\ \text{and}
\ \frakj_2:=\Spann{[\fetth(x_1),[\fetth(x_2),[R^N(y_1,y_2)|V]]]}{x_1,x_2,y_1,y_2\in T_pM}\;.
\end{align*}
Then, according to~\cite[Proof of Thm.~3]{J1},
we have $\frakh_+=\frakj_0+\frakj_2$\,. By means of Eq.~\eqref{eq:Even}, $\frakh_0\subset\frakh_+$ and,
furthermore, using Eq.~\eqref{eq:Degree1}-\eqref{eq:Degree2}, 
$R^N(x_j,x_k)|V$ and $[\fetth(x_j),[\fetth(x_k),[R^N(x_l,x_m)|V]]]$ are
homogeneous elements of degree $\delta_j+\delta_k$ and $\delta_j+\delta_k+\delta_l+\delta_m$\,, respectively, for all
$(x_j,x_k,x_l,x_m)\in D^j_p\times D^k_p\times D^l_p\times D^m_p$ and
$j,k,l,m=1,\ldots r$\,.  
Thus $\frakh_0$ is necessarily generated (as a vector space) by the following sets,
\begin{align}\label{eq:generators_0}
&S_{\!jj}:=\Menge{R^N(x_j,y_j)|V}{(x_j,y_j)\in D^j_p\times D^j_p\,}\,,\\
\label{eq:generators_2}
&S_{\!jkjk}:=\Menge{[\fetth(x_j),[\fetth(x_k),[R^N(y_j,y_k)|V]]]}{(x_j,y_k),(y_j,x_k)\in
  D^j_p\times D^k_p}\,,\\
\label{eq:generators_1}
&S_{\!jjkk}:=\Menge{[\fetth(x_j),[\fetth(y_j),[R^N(x_k,y_k)|V]]]}{(x_j,y_j)\in D^j_p\times D^j_p\,,(x_k,y_k)\in
  D^k_p\times D^k_p}
\end{align}
with $j,k=1,\ldots,r$\,.

Clearly, $S_{\!jj}$ is contained in to r.h.s.\ of~\eqref{eq:hol_plus} for all $j=1,\ldots,r$\,. Moreover,
according to Eq.~\eqref{eq:ZweiteAbleitung}, 
\begin{align*}
[\fetth(x_j),[\fetth(x_k),R^N(y_j,y_k)]]
\notag
=&-R^N(\fetth(x_j)\,\fetth(x_k)\,y_j,y_k)+R^N(y_j,\fetth(x_j)\,\fetth(x_k)\,y_k)\\&+R^N(h(x_j,y_j),h(x_k,y_k))
+R^N(h(x_k,y_j),h(x_j,y_k))\qmq{(on $V$)\;,}\\
[\fetth(x_j),[\fetth(y_j),R^N(x_k,y_k)]]
\notag
=&-R^N(\fetth(x_j)\,\fetth(y_j)\,x_k,y_k)+R^N(x_k,\fetth(x_j)\,\fetth(y_j)\,y_k)\\&+R^N(h(x_j,x_k),h(y_j,y_k))
+R^N(h(x_j,y_k),h(y_j,x_k))\qmq{(on $V$)}
\end{align*}
for all $(x_j,y_j)\in D^j_p\times D^j_p\,,(x_k,y_k)\in
  D^k_p\times D^k_p$\,;
hence, also using Eq.~\eqref{eq:SimpleRelation1} and~\eqref{eq:SimpleRelation2}, 
both $S_{\!jjll}$ and $S_{\!jljl}$ are contained in
r.h.s.\ of~\eqref{eq:hol_plus}. This finishes the proof of Eq.~\eqref{eq:hol_plus}.
 \end{proof}

\section{Proof of Theorem~\ref{th:MainResult}}
\label{se:6}
\textbf{\boldmath Throughout this section, we assume that $N$ is a symmetric space which is
  of compact or non-compact type, that $M$ is a simply connected symmetric
  space without Euclidian factor whose de\,Rham decomposition is given by
   $M_1\times\cdots\times M_r$ and that $f:M\to N$ is a parallel isometric immersion.} 
As before, we let $L_i$ denote the canonical foliation of $M$ and $D^i:=TL_i$
be the corresponding distribution for $i=1,\ldots,r$\,.
Keeping some $p\in M$ fixed, $T_pM$ and $D_i(p)$ both are seen
as subspaces of $T_{f(p)}N$ by means of $T_pf$\,.

Because the curvature tensor of $M$ is parallel and since $L_i(p)$ is totally
geodesic in $M$\,, the Theorem of Ambrose/Singer implies that the holonomy Lie
algebra of $L_i(p)$ is the subalgebra of $\so(D^i_p)$ which is given by
\begin{align}
\label{eq:hol(Leaf)}
\hol(L_i(p)):=\Spann{R^M(x,y)|D^i_p}{x,y\in D^i_p}
\end{align}
for $i=1,\ldots,r$\,. Furthermore, since $\Hol(L_i(p))$ is connected,
Lemma~\ref{le:CanonicalConnection}~(b) in combination with
Remark~\ref{re:Switch} yields:

\begin{lemma}\label{le:deRham}
$W_i$ is an irreducible $\hol(L_i(p))$-module for $i=1,\ldots,r$\,.
\end{lemma} 
We put $V:=\osc_pf$ and $\so(V)^i:=\so(V)_0\oplus\so(V)_i$
(see Eq.~\eqref{eq:unit} and~\eqref{eq:Degree0}); then $\so(V)^i$ is a $\Z/2\Z$-graded Lie algebra.
Furthermore, let $\frakh$ be the Lie algebra from Def.~\ref{de:HolonomyLieAlgebra} 
and $\frakh^i$ be its subalgebra which was introduced in
Thm.~\ref{th:hol}\,. Then there is the splitting
$\frakh^i=\frakh_0\oplus\frakh_i$ turning $\frakh^i$ into a graded subalgebra of $\so(V)^i$\,.

We consider the usual
positive definite scalar product on $\so(V)$\,,
\begin{equation}\label{eq:Trace}
\g{A}{B}:=-\trace(A\circ B)\;,
\end{equation}
and let $P_i:\so(V)\to\frakh^i$ denote the orthogonal
projection onto $\frakh^i$\,.

In accordance with Eq.~\eqref{eq:Degree1},
we introduce the linear map $A^i:D^i_p\to\so(V)^i$ given by
 \begin{equation}\label{eq:Difference1}
A^i(x):=\fetth(x)-P_i(\fetth(x))
\end{equation}
for each $x\in D^i_p$ and $i=1,\ldots,r$\,.

\begin{definition}\label{de:centralizer}
Let $V'$ be a Euclidian vector space. For each subset $X\subset\so(V')$ we
have the corresponding centralizer in $\so(V')$\,, 
\begin{equation}\label{eq:centralizer}
\frakc(X):=\Menge{A\in\so(V')}{\forall B\in X:A\circ B=B\circ A}\,.
\end{equation}
\end{definition}

\begin{lemma}\label{le:InjectiveOrZero}
We have
\begin{equation}\label{eq:OuterDerivation3}
A^i(x)\in\frakc(\frakh^i)\cap\so(V)_i
\end{equation}
for each $x\in D^i_p$ and $i=1,\ldots,r$\,. Furthermore, the following is
true:
Either $A^i$ is an injective map or $\fetth(x)\in\frakh_i$ holds for each $x\in D^i_p$\,. 
\end{lemma}
\begin{proof}
For Eq.~\eqref{eq:OuterDerivation3}:
By means of the splitting $\frakh^i=\frakh_0\oplus\frakh_i$ and
using Eq.~\eqref{eq:OuterDerivation1},\eqref{eq:Degree0} and~\eqref{eq:Degree1},
we can use analogous arguments as in~\cite[Proof of Prop.~11]{J2}.

For the last assertion, 
we use Lemma~\ref{le:deRham} in combination with Eq.~\eqref{eq:OuterDerivation1} and~\eqref{eq:hol_plus}
in order to apply similar arguments as in~\cite[Proof of Prop.~11]{J2}.
 \end{proof}

\begin{proposition}\label{p:Prop1001}
Suppose that $\dim(M_i)\geq 3$ and that the symmetric space $\bar M_i$ defined
by Eq.~\eqref{eq:tg} is irreducible (cf.\ Fn.~\ref{fn:irreducible}) for
$i=1,\ldots,r$\,. Then the following estimate is valid,
\begin{equation}
\label{eq:Estimate3}
\dim\big(\frakc(\frakh^i)\cap\so(V)_i\big)\leq 2\;.
\end{equation}
\end{proposition}
A proof of this Proposition will be given in Section~\ref{se:6.1}.

\paragraph{Proof of Thm.~\ref{th:MainResult}}
\begin{proof}
Let $f:M\to N$ be a parallel isometric immersion from the simply connected
Riemannian product space $M\cong M_1\times\cdots\times M_r$\,, where
$M_i$ is some irreducible symmetric space of dimension at least 3 for
$i=1,\ldots,r$\,.

For ``$(e)\Rightarrow (d)$'':  
Here we have $\dim(\frakc(\frakh^i)\cap\so(V)_i)<\dim(M_i)$\,, as a consequence of
Prop.~\ref{p:Prop1001}; hence Eq.~\eqref{eq:OuterDerivation3} implies that the
linear map $A^i$ defined by Eq.~\eqref{eq:Difference1} is
not injective. But then already $A^i=0$ according to
Lemma~\ref{le:InjectiveOrZero}, i.e.\ the availability of the relation
\begin{equation}\label{eq:fetth_in_hol2}
\forall x\in D^i_p:\fetth(x)\in\frakh_i
\end{equation}
is ensured for $i=1,\ldots,r$\,. This implies that Eq.~\eqref{eq:fetth_in_hol} holds.

``$(d)\Rightarrow (c)$'' follows from~\cite[Thm.~2 and Lemma~15]{J2}.

``$(c)\Rightarrow (b)$'' is trivial.

For ``$(b)\Rightarrow (a)$'':
Let $M$ be a symmetric space and $f:M\to N$ be an isometric immersion such
that there exists a connected Lie subgroup $G\subset\Iso(N)$ which
acts transitively on $f(M)$\,; in this situation, one can show that $f:M\to f(M)$ is a
covering map and there exists an equivariant Lie group homomorphism
$\tau:G\to\Iso(M)^0$ such that $\tau(G)$ acts transitively on $M$ (the proof
for this fact is standard and left to the reader).
We now see from Eq.~\eqref{eq:IsotropyGroup} that $\tau(G)$ acts transitively on $L_i(p)$ and therefore 
$f(L_i(p))$ is a homogeneous submanifold of $N$\,, too.

By contradiction, now we additionally assume that $f(L_i(p))$ 
is contained in some flat of $N$\,. Since $N$ is of compact or of
non-compact type, every homogenous submanifold of $N$ which is contained in some flat of
$N$\,, is intrinsically flat, too (cf.~\cite[Prop.~1]{J2}). Hence the
previous implies that $L_i(p)$ is a Euclidian space. Therefore, $M_i$ is also a
Euclidian space, which is contrary to our assumptions.

For ``$(a)\Rightarrow (e)$'': Since $f|L_i(p)$ is a parallel isometric
immersion defined on the simply connected, irreducible symmetric space
$L_i(p)$ (in accordance with Cor.~\ref{co:slice}) and, furthermore, $N$ is
of compact or non-compact type, this direction follows from~\cite[Thm.~5]{J2}.
 \end{proof}

\subsection{Proof of Proposition~\ref{p:Prop1001}}\label{se:6.1}
\textbf{\boldmath Besides the conventions made at the beginning of Section~\ref{se:6},
from now on we also assume that $m_i:=\dim(M_i)$ is at least $3$ and that the
symmetric space $\bar M_i$ defined by Eq.~\eqref{eq:tg} is irreducible (in
the sense of Fn.~\ref{fn:irreducible}) for 
$i=1,\ldots,r$\,.} In the following, we keep $p\in M$ fixed and abbreviate $W:=T_pM$ and
$W_i:=D^i_p$\,; then $W$ and $W_i$ both are seen as subspaces of $T_{f(p)}N$\,. 

In this situation, since $\bar M_i$ is totally geodesically embedded, 
$W_i$ is a curvature invariant subspace of $T_{f(p)}N$ and
$$
W_i\times W_i\times W_i\to W_i\;, (x,y,z)\mapsto R^N(x,y)\,z
$$
is the curvature tensor of $\bar M_i$ at $p$ (since $\bar M_i$ is totally
geodesic) for $ i=1,\ldots,r$\,. Analogous to Eq.~\eqref{eq:hol(Leaf)} and
Lemma~\ref{le:deRham} we have:

\begin{lemma}\label{le:hol(barMi)}
The holonomy Lie
algebra of $\bar M_i$ is the subalgebra of $\so(W_i)$ which is given by
\begin{equation}\label{eq:hol(barMi)}
\hol(\bar M_i):=\Spann{R^N(x,y)|W_i}{x,y\in W_i}\;.
\end{equation}
Furthermore, $W_i$ is an irreducible $\hol(\bar M_i)$-module for $i=1,\ldots,r$\,.
\end{lemma}

In accordance with Eq.~\eqref{eq:FirstNormalBundle} and~\eqref{eq:FirstNormalBundle2},
we also set $U:=\bbF_p$\,, $U_{ij}:=\bbF^{ij}_p$ for $i\neq j$ 
and $\tilde U:=\tilde\bbF_p$ \,; then $V:=\osc_pf=W\oplus U$ is the second osculating space
of $f$ at $p$\,. Now we introduce the following linear spaces $V_{ij}$ for $i,j=1,\ldots,r$ \,,
 \begin{align}\label{eqVii}
&V_{ii}:=W_i\oplus\tilde U\;,\\
\label{eqVij}
&V_{ij}:=W_i\oplus U_{ij}\ \ \text{for}\ i\neq j\;.
\end{align} 
Then we have $A(V_{ij})\subset(V_{ij})$ for each
$A\in\frakh^i$ and $i,j=1,\ldots,r$\,.
Hence we can consider the corresponding centralizers $\frakc(\frakh^i|V_{ij})\subset\so(V_{ij})$ (cf. Def.~\ref{de:centralizer})\,.

In order to prove Eq.~\eqref{eq:Estimate3}, we will additionally need the following two
estimates for $i=1,\ldots,r$\,,
\begin{align}
\label{eq:Estimate1}
&\frakc(\frakh^i|V_{ij})\cap\so(V_{ij})_-=\{0\}\ \ \text{for}\ i\neq j\;,\\
\label{eq:Estimate2}
&\dim\big(\frakc(\frakh^i)|V_{ii})\cap\so(V_{ii})_-\big)\leq 2\;.
\end{align}

A proof of these equations is given in Section~\ref{se:6.2} and~\ref{se:6.3}. 

\paragraph{Proof of Proposition~\ref{p:Prop1001}}
\begin{proof}
Recall that 
\begin{equation}\label{eq:Blockform0}
U=\tilde U\oplus\bigoplus_{i\neq j}U_{ij}
\end{equation}
is an orthogonal sum decomposition, according to
Thm.~\ref{th:decomposition}. Therefore, and because of Eq.~\eqref{eq:unit}
and~\eqref{eq:Degree0}, we have
\begin{align}
\label{eq:Blockform1}
&\bigoplus_{j=1,\ldots,r}V_{ij}\subset V\ \text{(as an orthogonal sum) and}\
A(V_{ij})\subset V_{ij}\ \text{for each}\ A\in\so(V)^i\;.
\end{align}
Moreover, the splitting of vector spaces~\eqref{eqVii},\eqref{eqVij} induces the splitting
$\so(V_{ij})=\so(V_{ij})_+\oplus\so(V_{ij})_-$ (as described in
Remark~\ref{re:finer}) such that, as a consequence of Eq.~\eqref{eq:Degree0},
\begin{equation}\label{eq:Blockform2}
\so(V)_i\to\bigoplus_{j=1,\ldots,r}\so(V_{ij})_-\,,\ A\mapsto
\bigoplus_{j=1,\ldots,r}A|V_{ij}\ \text{is an isomorphism\;.}
\end{equation}
Then Eq.~\eqref{eq:Blockform1} and~\eqref{eq:Blockform2} imply that
$$
\frakc(\frakh^i)\cap\so(V)_i\cong\bigoplus_{j=1}^r \frakc(\frakh^i|V_{ij})\cap\so(V_{ij})_-\;;
$$
now Eq.~\eqref{eq:Estimate3} follows from Eq.~\eqref{eq:Estimate1} and~\eqref{eq:Estimate2}.
 \end{proof}

\subsection{Proof of Equation~\eqref{eq:Estimate1}} 
\label{se:6.2}
For this, we keep a pair $(i,j)$ with $i\neq j$ fixed, let $\bbF^{ij}$ be 
the vector bundle defined by Eq.~\eqref{eq:FirstNormalBundle1} 
and set $U:=\bbF_p$\,, $W_i:=D^i_p$\,, $W_j:=D^j_p$ and $U_{ij}:=\bbF^{ij}_p$\,. Then
$V_{ij}:=W_i\oplus U_{ij}$ is the linear space defined by Eq.~\eqref{eqVij}.
We let $K$ denote the isotropy subgroup of $\Iso(N)^0$ at $p$\,, $\rho:K\to T_pM$ be the
corresponding isotropy representation and $\alpha:K\times\bbF\to\bbF$ be the action
described by Eq.~\eqref{eq:alpha}. 
 
Let $K\cong K_1\times\cdots\times K_r$ be the induced product
structure\,, see Eq.~\eqref{eq:IsotropyGroup}\,. Then $K_i\times K_j$ acts
orthogonally on $U_{ij}$ via $\alpha$ (see again Eq.~\eqref{eq:FirstNormalBundle1}
and Eq.~\eqref{eq:alpha}).

\begin{proposition}\label{p:possibilities}
Only one of the following three cases can occur:
\begin{itemize}
\item $U_{ij}=\{0\}$\,.
\item The dimension of $U_{ij}$ is at least $1/2\,\,m_i\,m_j$ and $K_i\times
  K_j$ acts irreducibly on $U_{ij}$\,.
\item There is the splitting $U_{ij}=U'\oplus U''$ into two irreducible 
invariant subspaces of the same dimension $1/2\,m_i\,m_j$\,.
\end{itemize}
\end{proposition}
\begin{proof}
According to Lemma~\ref{le:ExteriorTensorprodukt}, the induced action of
$K_i\times K_j$ on $W_i\otimes W_j$ is either irreducible or $W_i\otimes W_j$ 
is the direct sum of two irreducible invariant subspaces of the same dimension
 $1/2\,m_i\,m_j$\,. 
Furthermore, $h_p:W_i\otimes W_j\to U_{ij}$ is a surjective homomorphism. Now the ``Lie group version'' of
Lemma~\ref{le:Schur2}~(d) (cf. Remark~\ref{re:Reformulate}) implies that either $U_{ij}$ is also
the direct sum of two irreducible invariant subspaces of the
same dimension $1/2\,m_i\,m_j$ or $U_{ij}$ is irreducible
and of dimension at least $1/2\,\,m_i\,m_j$\,. 
 \end{proof}

We recall that $\bbF^{ij}_q$ is a curvature invariant subspace of $T_{f(q)}M$
for each $q\in M$ (Cor.~\ref{co:CurvatureInvariant}). Hence, on the one hand,
as $q$ varies over $M$\,,
 \begin{equation}\label{eq:Tij}
T:\bbF^{ij}_q\times\bbF^{ij}_q\times\bbF^{ij}_q\to\bbF^{ij}_q\;,(\xi_1,\xi_2,\xi_3)\mapsto R^N(\xi_1,\xi_2)\,\xi_3
\end{equation}
defines a section of the induced vector bundle $\rmL^3(\bbF^{ij};\bbF^{ij})$\,. Let
the latter be equipped with the connection which is canonically induced by $\nabla^\bot$\,.

\begin{lemma}\label{le:T}
$T$ is a parallel section.
\end{lemma}
\begin{proof}
For any differentiable curve $c:\R\to M$ and all sections 
$\xi_1,\ldots,\xi_4$ of $\bbF^{ij}$ along $c$\,, the function
$f(t):=\g{R^N(\xi_1(t),\xi_2(t))\,\xi_3(t)}{\xi_4(t)}$ is constant according
to~\cite[Prop.~8]{J1}. The result follows.
\end{proof}

On the other hand, using arguments given already in Sec.~\ref{se:2.6},
we now see that 
\begin{equation}
\label{eq:Mij}
\tilde M:=\exp^N(U_{ij})\subset N
\end{equation} 
is also a totally geodesically embedded symmetric space and that $T_p$ is the
curvature tensor of $\tilde M$ at $p$\,. 
Then, analogous to Eq.~\eqref{eq:hol(barMi)}, the holonomy Lie algebra of $\tilde M$
is the subalgebra of $\so(U_{ij}$ which is given by
\begin{equation}\label{eq:HolonomyOfFij}
\hol(\tilde M):=\Spann{R^N(\xi_1,\xi_2)|U_{ij}}{\xi_1,\xi_2\in U_{ij}}\;.
\end{equation}

Let $U_{ij}=U_0\oplus\cdots \oplus U_k$ be an orthogonal
decomposition such that $U_0$ is the largest vector subspace of
$U_{ij}$ on which $\hol(\tilde M)$ acts trivially and that $U_l$ is an irreducible $\hol(\tilde M)$-module for $l=1,\ldots,k$\,.
By virtue of the de\,Rham Decomposition Theorem, there exists a Euclidian
space $\tilde M_0$ and irreducible symmetric spaces $\tilde M_l$ ($l=1,\ldots,k$) such that the universal covering
space of $\tilde M$ is isometric to the Riemannian product of $\tilde M_0\times \cdots\times
\tilde M_k$ and that $U_l\cong T_p\tilde M_l$ for $l=0,\ldots,k$\,. Moreover, this de\,Rham decomposition of $\tilde M$ (and
hence the linear spaces $U_l$ are also unique (up to isometry, respectively,
and a permutation of $\{\tilde M_1,\ldots,\tilde M_k\}$).

\begin{corollary}\label{co:four_cases}
\begin{enumerate}
\item $U_l$ is invariant under the action of $K_i\times K_j$ on $U_{ij}$ for $l=0,\ldots,k$\,.
\item There are no more than the following five possibilities:
\begin{itemize} 
\item $U_{ij}$ is trivial.
\item $\tilde M$ is an irreducible symmetric space of dimension at least $1/2\,m_i\,m_j$\,.
\item $\tilde M$ is a flat of $N$\,.
\item
 The universal covering space of $\tilde M$ 
is isometric to the Riemannian product $\tilde M_0\times \tilde M_1$ where
$\tilde M_0$ is Euclidian and $\tilde M_1$ is an irreducible symmetric space 
such that $\dim(\tilde M_0)=\dim(\tilde M_1)=1/2\,m_i\,m_j$\,.
\item
 The universal covering space of $\tilde M$ is isometric to the Riemannian product
 $\tilde M_1\times \tilde M_2$ of two irreducible symmetric spaces satisfying $\dim(\tilde M_1)=\dim(\tilde M_2)=1/2\,m_i\,m_j$\,.
\end{itemize}
\end{enumerate}
\end{corollary}
\begin{proof}
For~(a): Combining the two facts that the $\nabla^\bot$-holonomy group of
$\bbF$ is given by Eq.~\eqref{eq:54321} (see Lemma~\ref{le:CanonicalConnection}~(c)) and that $T$ is a $\nabla^\bot$-parallel
section (see Lemma~\ref{le:T}), we get that $T(g\,\xi_1,g\,\xi_2)\,g\,\xi_3=g\,T(\xi_1,\xi_2)\,\xi_3$ for all
$\xi_1,\xi_2,\xi_3\in U_{ij}$ and $g\in K$\,. Hence $g(U_0)$ is also a
trivial $\hol(\tilde M)$-module and $\hol(\tilde M)$ acts irreducibly on $g(U_l)$ for
each $g\in K$ and $l=1,\ldots,k$\,, too, in accordance with Eq.~\eqref{eq:HolonomyOfFij}.
Also using the uniqueness of the de\,Rham decomposition, a continuity argument now shows that $g(U_l)=U_l$
for each $g\in K$ and $l=0,\ldots,k$\,. Our result follows.

(b) is now an immediate consequence of Prop.~\ref{p:possibilities}
and the uniqueness assertion of Lemma~\ref{le:Schur2}~(c) (again, we use its ``Lie group version''.
 \end{proof} 

In the following, we let $\frakh_0$ be the Lie algebra described in
Thm.~\ref{th:hol} and $\frakg$ be any subalgebra of $\frakh_0$\,. 
Then Eq.~\eqref{eq:unit} implies that there are induced representations,
\begin{align}\label{eq:rho_1}
&\rho_j:\frakg\to\so(W_j),\;A\mapsto
A|W_j:W_j\to W_j\;,\\
\label{eq:rho_2}
&\rho_{ij}:\frakg\to\so(U_{ij}),\;A\mapsto
A|U_{ij}:U_{ij}\to U_{ij}
\end{align}
for each $j\neq i$\,. Hence the linear space $\Hom_\frakg(W_j,U_{ij})$ is
defined in accordance with Def.~\ref{de:Hom}.

\begin{lemma}\label{le:Hom2}
The natural isomorphism  $\so(V_{ij})_-\to\rmL(W_j,U_{ij})$ 
provided by Lemma~\ref{le:splitting} induces the inclusion
\begin{align}\label{eq:Hom3}
&\frakc(\frakh^i|V_{ij})\cap\so(V_{ij})_-\hookrightarrow \Hom_\frakg(W_j,U_{ij})\;.
\end{align}
\end{lemma}
\begin{proof}
Since $\frakg\subset\frakh^i$ is a subalgebra\,, we have $[A,\frakg]=\{0\}$
for each $A\in\frakc(\frakh^i)$\,. Now  it is
straightforward to show that
$A|W_j:W_j\to U_{ij}$ 
belongs to $\Hom_\frakg(W_j,U_{ij})$  
for each $A\in\frakc(\frakh^i)\cap\so(V)_i$ \,.
 \end{proof}

We also recall the following result (cf.~\cite[Prop.~7]{J2}, ):

\begin{lemma}\label{le:curved_flat}
Let a linear subspace $V'\subset T_pN$ be given. Then the following assertions are equivalent:
\begin{enumerate}
\item $V'$ is a curvature isotropic subspace of $T_pN$ (see Def.~\ref{de:curvature_isotropic})\,.
\item $\exp^N(V')$ is a flat of $N$\,.
\item The sectional curvature of $N$ vanishes on every 2-plane of $V'$\,, i.e. 
$\g{R^N(v,w)\,w}{v}=0$ for all $v,w\in V'$\,.
\end{enumerate}
\end{lemma}

\paragraph{Proof of Eq.~\eqref{eq:Estimate1}}
\begin{proof}
According to Lemma~\ref{le:Hom2}, it suffices to show that
$\Hom_\frakg(W_j,U_{ij})=\{0\}$ for some subalgebra $\frakg\subset\frakh_0$\,.
By means Cor.~\ref{co:four_cases}, it suffices to distinguish the
following five cases:

\textbf{Case 1:} $\tilde M$ is trivial.
Here we have $U_{ij}=\{0\}$\,, hence $\so(V_{ij})_-=\{0\}$\,; then our result is obvious.

\textbf{Case 2:} $ \tilde M$ is an irreducible symmetric space of dimension at
least $1/2\,m_i\,m_j$\,. Here we consider the subalgebra of $\frakh_0$ which is given by
$$\frakg:=\Spann{R^N(\xi,\eta)|V}{\xi,\eta\in U_{ij}}\;,$$
see Cor.~\ref{co:CurvatureInvariant} and Thm.~\ref{th:hol}.

First, we note that $\hol(\tilde M)$ acts irreducibly on $U_{ij}$\,, as a consequence of the de\,Rham Decomposition
Theorem. This together with Eq.~\eqref{eq:HolonomyOfFij} implies that the action
of $\frakg$ on $U_{ij}$ is irreducible. Therefore, and since, by
assumption, $1/2\,m_i\,m_j\geq 3/2\,m_j>m_j$\,, Lemma~\ref{le:Schur1}
implies that $\Hom_\frakg(U_{ij},W_j)=\{0\}$ holds. Thus $\Hom_\frakg(W_j,U_{ij})$ is also trivial, according to Lemma~\ref{le:Schur2}~(e).

\textbf{Case 3:} $\tilde M$ is a Euclidian space. Here we consider the subalgebra of $\frakh_0$ which is given by
$$\frakg:=\Spann{R^N(x,y)|V}{x,y\in W_j}\;,$$ see Cor.~\ref{co:slice} and Thm.~\ref{th:hol}.

One the one hand, $U_{ij}$ is even a curvature isotropic subspace of
$T_{f(p)}N$\,, as a consequence of Lemma~\ref{le:curved_flat} in combination
with Eq.~\eqref{eq:Mij}. Therefore, we have $\g{R^N(x,y)\,\xi}{\eta}=\g{R^N(\xi,\eta)\,x}{y}=0$ for all
$x,y\in W_j$ and $\xi,\eta\in U_{ij}$\,,  i.e. $\frakg$ acts trivially on $U_{ij}$\,.
On the other hand, $\frakg$ acts non-trivially and irreducibly on $W_j$\,,
according to Lemma~\ref{le:hol(barMi)}. Using now Lemma~\ref{le:Schur1}, we now conclude that $\Hom_\frakg(W_j,U_{ij})=\{0\}$\,.

\textbf{Case 4:} The universal covering space of $\tilde M$ is isometric to
    the Riemannian product $\tilde M_0\times \tilde M_1$ where $\tilde M_0$ is Euclidian and $\tilde M_1$ 
is an irreducible symmetric space such that $\dim(\tilde M_0)=\dim(\tilde M_1)=1/2\,m_i\,m_j$\,.
Here we set $\frakg:=\frakh_0$\,. Let $\lambda\in\Hom_\frakg(W_j,U_{ij})$ be given. I claim that
$\lambda=0$\,.

For this: Put $U_0:=T_p\tilde M_0$ and $U_1:=T_p\tilde M_1$\,. Then, following the
arguments from Case~2, the linear space
$$\tilde\frakg:=\Spann{R^N(\xi,\eta)|V}{\xi,\eta\in U_1}$$ 
is actually a subalgebra of $\frakg$ and $U_1$ is a $\tilde\frakg$-invariant subspace of
$U_{ij}$ such that $\Hom_{\tilde\frakg}(W_j,U_1)=\{0\}$\,. 
Hence Lemma~\ref{le:Schur2}~(f) implies that $\lambda(W_j)\subset U_0$\,.

Then, repeating the arguments given in Case~3, we see that $\bar\frakg:=\Spann{R^N(x,y)|V}{x,y\in
  W_j}$ is a subalgebra of $\frakh_0$ and $\lambda(W_j)$ is a
$\bar\frakg$-invariant subspace of $U_{ij}$ such that
$\Hom_{\bar\frakg}(W_j,\lambda(W_j))=\{0\}$\,. We thus conclude that $\lambda=0$ and
therefore $\Hom_\frakg(W_j,U_{ij})=\{0\}$\,.

\textbf{Case 5:} The universal covering space of $\tilde M$ is isometric
    to the Riemannian product $\tilde M_1\times \tilde M_2$ of two
    irreducible symmetric spaces satisfying $\dim(\tilde M_1)=\dim(\tilde M_2)=1/2\,m_i\,m_j$\,.
Again we put $\frakg:=\frakh_0$\,. Let $\lambda\in\Hom_\frakg(W_j,U_{ij})$ be given. I claim that
$\lambda=0$\,.

For this: Put $U_1:=T_p\tilde M_1$ and $U_2:=T_p\tilde M_2$\,. 
Then, repeating an argument from Case~4, we conclude that $\lambda(W_j)\subset U_2$\,.
Vice versa, we have $\lambda(W_j)\subset U_1$\,; therefore, $\lambda=0$\,. 
\end{proof}
\subsection{Proof of Equation~\eqref{eq:Estimate2}}
\label{se:6.3}
For this, we keep $i\in\{1,\ldots,r\}$ fixed, put $f_i:=f|L_i(p):L_i(p)\to N$
and let $\osc f|L_i(p)$ and $f_i^*TN$ denote
the corresponding  pullback bundles; then $\osc f|L_i(p)$ and $f_i^*TN$
both are vector bundles over $L_i(p)$\,. Moreover, $\osc f|L_i(p)\subset
f_i^*TN$ is a parallel subbundle; therefore, we can repeat the construction
from Section~\ref{se:2.5} to obtain the corresponding Holonomy group
$\Hol(\osc f|L_i(p))$ with respect to the connection induced by $\nabla^N$\,.

\begin{definition}
Let $\tilde\frakh^i$ denote the Lie algebra of $\Hol(\osc f|L_i(p))$\,.
\end{definition} 
Furthermore, we put $V:=\osc_pf$ and let $\frakh$ be the subalgebra of
$\so(V)$ introduced in Def.~\ref{de:HolonomyLieAlgebra}. 
Clearly, $\tilde\frakh^i$ is a subalgebra of $\frakh$\,, hence $\tilde\frakh^i\subset\frakh\subset\so(V)$ is a sequence of
subalgebras.

\begin{proposition}\label{p:666}
$\tilde\frakh^i$ is a graded subalgebra of $\frakh$\,, i.e.
\begin{equation}\label{eq:99}
\tilde\frakh^i=\bigoplus_{\delta\in \scrA}\tilde\frakh^i\cap\so(V)_{\delta}\;.\end{equation}
Hence $\tilde\frakh^i\subset\frakh\subset\so(V)$ is actually a sequence of
$\scrA$-graded subalgebras.
\end{proposition}
\begin{proof}
Let $\Sigma_p^j$ denote the symmetries of $\osc f$ described in
Thm.~\ref{th:Symmetries} for $j=1,\ldots ,r$\,. Then the base map of $\Sigma_p^j$ is the map
$\sigma^j_p$ described by Eq.~\eqref{eq:Basemap}; in particular,
$\sigma^j_p(L_i(p))=L_i(p)$ holds for all $j$\,. Therefore, $\Sigma_p^j$ induces a parallel
vector bundle involution on $\osc f|L_i(p)$ along the geodesic symmetry of
$L_i(p)$ (for $j=i$) or the identity map of $L_i(p)$ (in case
$j\neq i$). Now we can use the arguments from the proof of
Thm.~\ref{th:hol}~(a) to deduce our result.
\end{proof}

Let $\frakh^i$ be the subalgebra of $\frakh$ which was introduced in
Thm.~\ref{th:hol}; then $\frakh^i$ is also a subalgebra of $\so(V)^i:=\so(V)_0\oplus\so(V)_i$\,.

\begin{lemma}\label{le:hilf2}
\begin{enumerate}
\item The pullback bundle $D^j\oplus\bbF^{ij}|L_i(p)$ is a parallel subbundle of $\osc f|L_i(p)$\,. The same is true for $D^i\oplus\tilde\bbF|L_i(p)$\,.
\item 
$\tilde\frakh^i$ is already a subalgebra of $\frakh^i$\,. Hence the splitting Eq.~\eqref{eq:99} gets actually reduced to 
\begin{equation}
\tilde\frakh^i=\tilde\frakh^i\cap\so(V)_0\oplus\tilde\frakh^i\cap\so(V)_i\;,
\end{equation}
and $\tilde\frakh^i\subset\frakh^i\subset\so(V)^i$ is a sequence
of $\Z/2\Z$-graded subalgebras.
\end{enumerate}
\end{lemma}
\begin{proof}
For~(a): Use the equations of Gau{\ss} and Weingarten in combination with Lemma~\ref{le:hilf1}.

For~(b): Because of Part~(a) combined with Eq.~\eqref{eq:unit} and~\eqref{eq:Degree0}, we have $\tilde\frakh^i\subset
\so(V)^i$\,. The result immediately follows from Eq.~\eqref{eq:99}.
 \end{proof}

We set $W_i:=D^i_p$ and $\tilde U:=\tilde\bbF_p$\,; then $W_i\oplus
\tilde U$ is the linear space $V_{ii}$ defined by Eq.~\eqref{eqVii}.
By means of the previous lemma, we have $A(V_{ii})\subset V_{ii}$ for each
$A\in\tilde\frakh^i$ and hence we can introduce the corresponding centralizer $\frakc(\tilde\frakh^i|V_{ii})\subset\so(V_{ii})$
(cf. Def.~\ref{de:centralizer}). In particular, Part~(b) of
Lemma~\ref{le:hilf2} yields:

\begin{corollary}\label{co:hilf1}
We have 
\begin{equation}\label{eq:Inclusion1}
\frakc(\frakh^i|V_{ii})\subset\frakc(\tilde\frakh^i|V_{ii})\;.
\end{equation}
\end{corollary}

As mentioned already before, the vector space
$$
\frakg:=\Spann{R^N(x,y)|V}{x,y\in W_i}$$
is a subalgebra of $\so(V)_0$\,.
Thus Eq.~\eqref{eq:unit} also implies that there are induced representations, 
\begin{align}\label{eq:rho_3}
&\rho_i:\frakg\to\so(W_i),\;A\mapsto
A|W_i:W_i\to W_i\;,\\
\label{eq:rho_4}
&\tilde\rho:\frakg\to\so(\tilde U),\;A\mapsto
A|\tilde U:\tilde U\to \tilde U\;.
\end{align}

We introduce the vector space $\Hom_\frakg(W_i,\tilde U)$ according to
Def.~\ref{de:Hom} and consider the splitting
$\so(V_{ii})=\so(V_{ii})_+\oplus\so(V_{ii})_-$ induced by the splitting
$V_{ii}=W_i\oplus \tilde U$\,. By means of the Theorem of Ambrose/Singer, 
$\frakg\subset\tilde\frakh^i$\,; hence we have (analogous to Lemma~\ref{le:Hom2})

\begin{lemma}\label{le:Hom3}
The natural isomorphism  $\so(V_{ii})_-\to\rmL(W_i,\tilde U)$ 
provided by Lemma~\ref{le:splitting} induces the inclusion
\begin{align}\label{eq:Hom4}
&\frakc(\tilde\frakh^i|V_{ii})\cap\so(V_{ii})_-\hookrightarrow \Hom_\frakg(W_i,\tilde U)\;.
\end{align}
\end{lemma}

Set $U_{ii}:=\bbF^{ii}_p$\,; then $U_{ii}$ is a subspace of $\tilde U$\,.

\begin{lemma}\label{le:hilf3}
Let $U_{ii}^\bot$ denote the orthogonal complement of $U_{ii}$ in $\tilde U$\,.
\begin{enumerate}
\item
$\frakg$ acts trivially on $U_{ii}^\bot$\,.
\item We have $\lambda(W_i)\subset U_{ii}$ for each $\lambda\in\Hom_\frakg(W_i,\tilde U)$\,.
\end{enumerate}
\end{lemma}
\begin{proof}
For~(a): Let $\xi\in U_{ii}^\bot$ be given. Using Lemma~\ref{le:hilf1}, we immediately obtain that $S_\xi\,x=0$ for all $x\in W_i$\,. Hence
\begin{equation}\label{eq:592}
\forall x,y\in W_i:\;[\fetth(x),\fetth(y)]\,\xi=0\;.
\end{equation}

Because of Eq.~\eqref{eq:592} and the Equations of Gau{\ss}, Codazzi and
Ricci~\eqref{eq:Gauss-Ricci}, it thus remains to show that
\begin{equation}\label{eq:594}
R^\bot(W_i\times W_i)\,\xi=\{0\}\;.
\end{equation}

For this: Let $\bbF^\sharp$ denote the maximal flat subbundle of
$\bbF$ and set $U^\sharp:=\bbF^\sharp_p$\,.
As a consequence of Thm.~\ref{th:decomposition}, $\xi$ belongs to the linear space
\begin{equation}
\label{eq:593}
U^\sharp\oplus\sum_{j\neq i}U_{jj}\;.
\end{equation}
In case $\xi\in U^\sharp$\,, Eq.~\eqref{eq:594} is trivial.
In case $\xi\in U_{jj}$ with $j\neq i$\,, we may assume that there exist $x_j,y_j\in
W_j$ with $\xi=h(x_j,y_j)$\,.
Then, according to~\cite[Prop.~4~(d)]{J1}, and since $M$ is the
Riemannian product of the $M_j$'s,
\begin{align*}
R^\bot(W_i\times W_i)\,\xi=h(R^M(W_i\times W_i)\,x_j,y_j)+h(x_j,R^M(W_i\times W_i)\,y_j)=\{0\}\;.
\end{align*}
Eq.~\eqref{eq:594} follows.

(b) is now a consequence of Part~(a) and arguments given already in the
proof of Eq.~\eqref{eq:Estimate1}.
 \end{proof}

Now recall that $f_i$ is a parallel isometric immersion, according to
Cor.~\ref{co:slice}; then $U_{ii}$ is the first normal space and
$V_i:=W_i\oplus U_{ii}$ is the second osculating space of $f_i$ at $p$\,; note
that $V_i$ is a subspace of $V_{ii}$\,.
Let $\osc f_i$ denote the second osculating bundle of $f_i$ and let us apply
Def.~\ref{de:HolonomyLieAlgebra} to define its extrinsic holonomy Lie algebra $\bar\frakh^i$\,; then
$\bar\frakh^i$ is a subalgebra of $\so(V_i)$\,. Furthermore, there is the
corresponding centralizer $\frakc(\bar\frakh^i)\subset\so(V_i)$ (cf. Def.~\ref{de:centralizer}).

\begin{corollary}\label{co:hilf2}
There exists an injective map
\begin{equation}\label{eq:Inclusion2}
\frakc(\tilde\frakh^i|V_{ii})\cap\so(V_{ii})_- \hookrightarrow\frakc(\bar\frakh^i)\cap\so(V_i)_-\;.
\end{equation}
\end{corollary}
\begin{proof}
Let $A\in\frakc(\tilde\frakh^i|V_{ii})\cap\so(V_{ii})_-$ be given and
$U_{ii}^\bot$ be defined as in Lemma~\ref{le:hilf3}. We will show that
\begin{align}\label{eq:hilf10}
&A(W_i)\subset U_{ii}\ \text{and}\ A(U_{ii})\subset W_i\;,\\
\label{eq:hilf100}
&A|U_{ii}^\bot=0\;,\\
\label{eq:hilf1000}
&[A|V_i,\bar\frakh^i]=\{0\}\;;
\end{align}
hence $A\mapsto A|V_i$ gives the desired map~\eqref{eq:Inclusion2}.

For Eq.~\eqref{eq:hilf10} and~\eqref{eq:hilf100}: On the one hand, we have
$A(\tilde U)\subset W_i$ (since
$A\in\so(V_{ii})_-$). On the other hand, $A|W_i\in\Hom_\frakg(W_i,\tilde U)$
according to Lemma~\ref{le:Hom3}, hence $A(W_i)\subset U_{ii}$ because of
Lemma~\ref{le:hilf3}. Now Eq.~\eqref{eq:hilf10} follows and, moreover,
$\g{A\,u}{w}=-\g{u}{A\,w}=0$ for all $u\in U_{ii}^\bot$\,, $w\in W_i$\,.
Thus Eq.~\eqref{eq:hilf100} also holds.

For Eq.~\eqref{eq:hilf1000}: Because $V_i$ is the second osculating space
of $f_i$ at $p$ and since, furthermore, $\osc f_i\subset\osc f|L_i(p)$ is a parallel
subbundle as mentioned already before, $B(V_i)\subset V_i$ for each
$B\in\tilde\frakh^i$ and, moreover, $\tilde\frakh^i\to\bar\frakh^i,\; B\mapsto B|V_i$
is a surjective map.
Furthermore, $[A,B]=0$ for all $B\in\tilde\frakh^i$\,. The previous together
with Eq.~\eqref{eq:hilf10} obviously implies Eq.~\eqref{eq:hilf1000}.
 \end{proof} 

\paragraph{Proof of Eq.~\eqref{eq:Estimate2}}
\begin{proof}
On the one hand, by means of Corollary~\ref{co:hilf1}
and~\ref{co:hilf2}, there exists a sequence of inclusions,
$$
\frakc(\frakh^i|V_{ii})\cap\so(V_{ii})_-\stackrel{\eqref{eq:Inclusion1}}\hookrightarrow\frakc(\tilde\frakh^i|V_{ii})\cap\so(V_{ii})_-
\stackrel{\eqref{eq:Inclusion2}}\hookrightarrow\frakc(\bar\frakh^i)\cap\so(V_i)_-\;.
$$
Hence $\dim(\frakc(\frakh^i|V_{ii})\cap\so(V_{ii})_-)\leq \dim(\frakc(\bar\frakh^i)\cap\so(V_i)_-)$\,.

On the other hand, by assumption, $f_i$ is a parallel isometric immersion
which is defined on the simply connected, irreducible symmetric space $L_i(p)$
such that $f_i(L_i(p))$ is not contained in any flat of $N$ and that
$\dim(L_i(p))\geq 3$\,. This situation was already investigated
in~\cite[Sec.~3.2]{J2}: According to Prop.~12 there, we have $\dim(\frakc(\bar\frakh^i)\cap\so(V_i)_-)\leq
2$\,. Now Eq.~\eqref{eq:Estimate2} follows.
 \end{proof}

\section*{Acknowledgements}
I want to thank Professor Ernst Heintze for a helpful discussion which lead to 
Theorem~\ref{th:decomposition}.

% \makeatletter
% \renewcommand\appendix{\par
%   \setcounter{section}{0}%
%   \setcounter{subsection}{0}%
%   \setcounter{equation}{0}
%   \gdef\thefigure{\@Alph\c@section.\arabic{figure}}%
%   \gdef\thetable{\@Alph\c@section.\arabic{table}}%
%   \gdef\thesection{\appendixname\@Alph\c@section}%
%   \@addtoreset{equation}{section}%
%   \gdef\theequation{\@Alph\c@section.\arabic{equation}}%
%   \addtocontents{toc}{\string\let\string\numberline\string\tmptocnumberline}{}{}
% }
% \def\appendixname{}
% \makeatother

\section*{Appendix}
\begin{appendix}
\section{Representation theory}
\label{se:Appendix}
Let $\frakk$ be a Lie algebra over $\R$\,, $V$ be a vector space
over a field $\bbK\in\{\R,\C\}$\, and $\rho:\frakk\to\gl_\bbK(V)$ be a representation.
Then we also say that ``$V$ is a $\frakk$-module''. $V$ is called
``irreducible'' if $\{0\}$ and $V$ are the only invariant subspaces of $V$\,; otherwise, $V$
is called ``reducible''.

\begin{definition}\label{de:Hom}
Let a second linear space $V'$ over $\bbK$ and a representation
$\rho':\frakk\to\gl_\bbK(V')$ be given. A linear map $\lambda:V\to
V'$ satisfying $\lambda(\rho(A)\,v)=\rho'(A)\,\lambda(v)$ for all
$A\in \frakk$ and $v\in V$ will be briefly called a ``homomorphism''. 
Then set of homomorphisms, denoted by $\Hom_\frakk(V,V')$\,, is a vector space over $\bbK$\,, too.
\end{definition}

The next lemma is ``standard''.

\begin{lemma}[Schur's Lemma]\label{le:Schur1}
If $\lambda:V\to V'$ is a homomorphism, then 
both the kernel and the image of $\lambda$ are invariant subspaces of $V$ and
$V'$\,, respectively. 
In particular, if $V$ is an irreducible $\frakk$-module and $\lambda\neq 0$\,,
then $\lambda$ is injective.
\end{lemma}

In the following, we will always assume that $V$ is a Euclidian space. The proof of the following lemma is left to the reader.

\begin{lemma}\label{le:Schur2}
\begin{enumerate}
\item Let invariant subspaces $U$ and $W$ of $V$ be given.
If  $U$ and $W$ both are irreducible, then we have $U\bot W$ unless $U\cong V$
(as $\frakk$-modules). 
Therefore, if $U$ is isomorphic  to a direct sum of non-trivial and irreducible $\frakk$-modules and
$W$ is a trivial $\frakk$-module, then $U\bot W$\,.
\item
We can always find a decomposition $V=W_1\oplus\cdots\oplus W_k$ into
pairwise orthogonal, invariant subspaces $W_i$ with the following property: There exists an
irreducible $\frakk$-module $W_i'$ and an integer $m_i$ such that $W_i$ is isomorphic to
the direct sum of $m_i$ copies of $W_i'$ for $i=1,\ldots,k$\,, and such that $\{W_1',\ldots,W_k'\}$ are pairwise
non-isomorphic (as $\frakk$-modules).  
\item 
The subspaces $W_i$\,, the ``multiplicities'' $m_i$ and the modules $W_i'$ are uniquely determined 
(the latter ones only up to isomorphy) for $i=1,\ldots,k$
(cf.~\cite[Ch.~XVIII, Prop.~1.2]{La}).
\end{enumerate}
Let a second Euclidian space $V'$ over $\bbK$ and a representation
$\tilde \rho:\frakk\to\gl_\bbK(V')$ be given. 
Furthermore,  let $V\cong \bigoplus_{j\in J}W_j$ be any decomposition into
irreducible submodules $W_j$\,.
\begin{enumerate}
\addtocounter{enumi}{3} 
\item 
If $\lambda:V\to V'$ is a surjective homomorphism, then
there exists a subset $\tilde J\subset J$ such that 
$\lambda|\bigoplus_{j\in\tilde J} W_j$ induces an isomorphism onto $V'$\,.
\item For every $\lambda\in\Hom_\frakk(V,V')$ the adjoint map
  $\lambda^*:V'\to V$ belongs to $\Hom_\frakk(V',V)$\,, and
  $\lambda\mapsto\lambda^*$ induces $\Hom_\frakk(V,V')\cong\Hom_\frakk(V',V)$\,.
\item Suppose that $V'=W\oplus U$ is the orthogonal sum of two invariant subspaces. If
  $\Hom_\frakk(V,W)=\{0\}$\,, then $\lambda(V)\subset U$ for each
  $\lambda\in\Hom_\frakk(V,V')$\,.
\end{enumerate}
\end{lemma} 

\begin{remark}\label{re:Reformulate}
Let a Lie group $K$\,, a vector 
space $V$ and a representation $\rho:K\to\rm\Gl_\bbK(V)$ be given. Then we also say that ``$K$
acts linearly $V$''. Similar as before, we define irreducible and reducible $K$-actions. 
Furthermore, there are corresponding ``Lie group versions'' of Lemma~\ref{le:Schur1} and
(in case $V$ is Euclidian and $K$ acts orthogonally on $V$) Lemma~\ref{le:Schur2}.
\end{remark}

Given a Lie algebra $\frakk$ over $\R$\,, a Euclidian space $V$ and a representation $\rho:\frakk\to\so(V)$\,, 
let $V^\C:=V\oplus \i V$ denote the  corresponding
``complexification'' and $\rho^\C:\frakk\to\gl_\C(V^\C)$ be the canonically induced representation.

\begin{lemma}\label{le:Komplexifikation}
\begin{enumerate}
\item 
Suppose that there exists $J\in\rmO(V)$ with $J^2=-\Id$ (equipping already $V$ with the
structure of a complex space) such that $\rho(\frakk)\subset\fraku(V)$\,. Then
\begin{equation}\label{eq:Vpm}
V_{\pm\i}:=\Menge{v\mp\i\,J\,v}{v\in V}\;.
\end{equation}
induces the decomposition $V^\C=V_{+\i}\oplus V_{-\i}$ into invariant subspaces such that 
\begin{equation}
\label{eq:IsomorphismOverC1}
V\to V_{+\i},\,v\mapsto 1/2\,(v-\i\,J\,v)
\end{equation}
is a complex linear isomorphism of $\frakk$-modules.

\item Conversely, if $V$ is irreducible, but $V^\C$ is reducible, 
then there exists necessarily some $J\in\rmO(V)$ with $J^2=-\Id$ such that $\rho(\frakk)\subset\fraku(V)$\,. 
\end{enumerate}
\end{lemma}

Let $\frakk_i$ be a Lie algebra over $\R$\,, $V_i$ be a vector space over $\bbK$ 
and $\rho_i:\frakk_i\to\gl_\bbK(V_i)$ be a representation for $i=1,2$\,.

\begin{definition}\label{de:ExteriorTensorprodukt1}
We set $\frakk:=\frakk_1\oplus\frakk_2$ and $V:=V_1\otimes_\bbK V_2$\,; then
$\frakk$ is a Lie algebra over $\R$ and $V$ is a vector space over $\bbK$\,.
Moreover, there is a natural representation $\rho_1\otimes\rho_2:\frakk\to\gl_\bbK(V)$\,,
given by 
\begin{equation}\label{eq:tensor1}
\big(\rho_1\otimes\rho_2\,(X_1+X_2)\big)\,v_1\otimes v_2:= \rho_1(X_1)\,v_1\otimes v_2+v_1\otimes \rho_2(X_2)\,v_2
\end{equation}
for all $X_1+X_2\in\frakk_1\oplus\frakk_2$ and $(v_1,v_2)\in V_1\times V_2$\,.
\end{definition}

\begin{lemma}\label{le:irreducible}
Let $\frakk_i$ be a semisimple Lie algebra (cf.~\cite[Ch.~1]{Kn}), $V_i$ be a Hermitian
vector space and $\rho_i:\frakk_i\to\fraku(V_i)$ be a complex-irreducible representation for
$i=1,2$\,. Then $\rho_1\otimes\rho_2$ is a complex-irreducible representation,
too.
\end{lemma}
\begin{proof}
Set $\frakk:=\frakk_1\oplus\frakk_2$ and $\rho:=\rho_1\otimes\rho_2$\,. 
Let $\frakg_i:=\frakk_i^\C$ denote the complexified Lie algebra, which is a
complex semisimple Lie algebra. Then $\rho_i$ induces an irreducible representation
$\rho_i:\frakg_i\to\gl(V_i)$\,. Moreover, if $\fraka_i$ is a maximal Abelian subspace of $\frakk_i$\,, then
$\fraka_i\oplus\i\,\fraka_i$ is a ``Cartan subalgebra'' of
$\frakg_i$ (see~\cite[Ch.~II]{Kn})\,. 
Let $\Delta_i$ denote the corresponding set of ``roots''. Then the roots
take real values on $\i\,\fraka$\,, and hence we may choose a ``total
ordering'' on the dual space of $\i\,\fraka$  to obtain
the corresponding set $\Delta^+_i$ of ``positive roots'' and the
corresponding ``highest weight'' $\lambda_i$ for $i=1,2$\,.
Furthermore, $\frakg:=\frakg_1\oplus\frakg_2$ is also a complex, semisimple Lie
algebra, $(\fraka_1\oplus\fraka_2)+\i\,(\fraka_1\oplus\fraka_2)$ is a Cartan subalgebra of $\frakg$ and 
$\Delta^{+}:=\Delta_1^{+}\dot\cup\Delta_2^{+}$ is the corresponding set of positive roots. 
Moreover, $\rho$ induces a representation of $\frakg$ on $V_1\otimes_\C
V_2$ whose weights are given by $\mu_1\oplus\mu_2$\,, 
where $\mu_1$ and $\mu_2$ range over the weights of $\rho_1$ and $\rho_2$\,,
respectively. In particular, $\lambda:=\lambda_1\oplus\lambda_2$ is the highest weight of $\rho$\,.
Let $V_\lambda$ be an irreducible $\frakg$-submodule of $V_1\otimes_\C V_2$ 
such that the corresponding Eigenspace $E_\lambda\subset
V_\lambda$ is non-trivial (such a module clearly exists). 
Then, according to the ``Theorem of the highest weight'' in
combination with ``Weyl's dimension formula'' (see~\cite[Ch.~5]{Kn}), we have 
$\dim(V_\lambda)=\dim(V_1)\cdot\dim(V_2)=\dim(V)$\,, hence $V_\lambda=V_1\otimes_\C V_2$ and
therefore already $V_1\otimes_\C V_2$ is an irreducible
$\frakk$-module.
\end{proof}

\begin{remark}\label{re:Switch}
Let $V$ be a Euclidian space, $K$ be a connected Lie group and
$\rho:K\to\rmO(V)$ be a representation. Let $\frakk$ be the Lie algebra of $K$
and $\rho:\frakk\to\so(V)$ be the induced representation.
Since $K$ is connected, a subspace $W\subset V$ is
$K$-invariant if and only if it is $\frakk$-invariant. 
Hence $K$ acts irreducible on $V$ if and only if $V$ is an irreducible $\frakk$-module.  
\end{remark}

Let a simply connected symmetric space $M$ be given. 
We let $K$ be the isotropy subgroup of $\Iso(M)^0$ at the origin $p$ and
$\rho:K\to\rmO(V)$ be the isotropy representation on $V:=T_pM$\,. 
Furthermore, let $\frakk$ be the Lie algebra of $K$ and $\rho:\frakk\to\so(V)$ be the induced representation.

\begin{lemma}\label{le:Hermitian}
Suppose that $M$ is an irreducible symmetric space. Then: 
\begin{enumerate}
\item $V$ is an irreducible $\frakk$-module and and $\rho(\frakk)$ is the Holonomy Lie algebra of $M$\,.
\item
$M$ is a Hermitian symmetric space if and only if 
there exists some $J\in\rmO(V)$ with $J^2=-\Id$ (equipping $V$ with the
structure of a unitary space) such that $\rho(\frakk)\subset\fraku(V)$\,. 
\item We always have $\frakk=[\frakk,\frakk]\oplus\frakc$\,, where $\frakc$ denotes the
center of $\frakk$\,. Furthermore, the commutator ideal $[\frakk,\frakk]$ is semisimple. 
\item 
In case $M$ is Hermitian, $\frakc$ is 1-dimensional and $\rho(\frakc)=\R\,J$
holds (see~(b)).
Otherwise, $\frakc=\{0\}$ and hence $\frakk$ is semisimple. 
\end{enumerate}
\end{lemma}
\begin{proof}
For (a) and (b): Since $K$ is connected (cf. Sec~\ref{se:3}), the result
follows from Lemma~\ref{le:CanonicalConnection}~(b) in combination with Remark~\ref{re:Switch}.

For~(c): Since $\rho$ is a
faithful representation, $\frakk$ can be seen as a subalgebra of
$\so(T_pM)$\,. Then for any ideal $\fraka\subset\frakk$ its orthogonal
complement $\fraka^\bot$ (see~\eqref{eq:Trace}) is an ideal of $\frakk$\,, too. 
Hence there exists a decomposition of $\frakk$ into an Abelian subspace and
simple ideals. Now (c) is obvious. 

For~(d): According to~\cite[Ch.~VIII,~\S~7]{He},
$K$ has non-discrete center $\rmZ_K$ if and only if $M$ is Hermitian and, in the
Hermitian case, $\rho$ maps $\rmZ_K$ isomorphically onto the circle subgroup $\rmS^1\subset\rmU(V)$\,. Now (d) follows from (c).
\end{proof}

Let $K_i$ be a Lie group, $V_i$ be a Euclidian space and $\rho_i:K_i\to\rmO(V_i)$ be a representation for
$i=1,2$\,.

\begin{definition}\label{de:ExteriorTensorprodukt2}
We let $K:=K_1\times K_2$ denote the product Lie
group and set $V:=V_1\otimes V_2$\,.
Then there is a natural representation $\rho_1\otimes\rho_2:K\to\Gl(V)$\,,
given by 
\begin{equation}\label{eq:tensor2}
\big(\rho_1\otimes\rho_2\,(g_1,g_2)\big)\,v_1\otimes v_2= \rho_1(g_1)\,v_1\otimes \rho_2(g_2)\,v_2
\end{equation}
for all $(g_1,g_2)\in K_1\times K_2$ and $(v_1,v_2)\in V_1\times V_2$\,.
\end{definition}

Let simply connected symmetric spaces $M_1$ and $M_2$ be given. Let $p_i$
be an origin of $M_i$\,, $K_i$ be the corresponding isotropy subgroup of
$\Iso(M_i)^0$ and $\rho_i:K_i\to\rmO(V_i)$ be the isotropy representation on $V_i:=T_{p_i}M_i$ for $i=1,2$\,. 

\begin{lemma}
\label{le:ExteriorTensorprodukt}
Suppose that $M_1$ and $M_2$ both are irreducible symmetric spaces. Let
$\rho_1\otimes\rho_2$ be defined according to Def.~\ref{de:ExteriorTensorprodukt2}. Then:
\begin{enumerate}
\item In case neither $M_1$ nor $M_2$ is a Hermitian symmetric space, $\rho_1\otimes\rho_2$
  is irreducible.
\item The same is true if exactly one of $M_1$ or $M_2$ is a Hermitian symmetric space.
\item In case $M_1$ and $M_2$ both are Hermitian symmetric spaces, we let $J_i$
  denote the corresponding complex structure of $V_i$ for $i=1,2$\,. Then 
\begin{align}
V_\pm:=\Spann{v_1\otimes v_2\mp J_1\,v_1\otimes J_2\,v_2}{(v_1,v_2)\in V_1\times
  V_2}
\end{align}
gives the decomposition $V_1\otimes V_2=V_+\oplus V_-$ into two irreducible invariant
subspaces of the same dimension. 
\end{enumerate}
\end{lemma}

\begin{proof}
Set $K:=K_1\times K_2$\,, $V:=V_1\otimes V_2$ and $\rho:=\rho_1\otimes\rho_2$\,. 
Since $K_i$ is connected, we may switch to the corresponding Lie algebras
$\frakk_i$ and their induced representations $\rho_i:\frakk_i\to\so(V_i)$
according to Remark~\ref{re:Switch}. Then $\rho_1\otimes\rho_2$ is given by
Eq.~\eqref{eq:tensor1} (because of the ``chain rule'').

For~(a): As a consequence of Lemma~\ref{le:Komplexifikation} in combination with Lemma~\ref{le:Hermitian}, 
in this case $\frakk_i$ acts irreducibly on $V_i^\C$ for $i=1,2$\,, too. Hence $V^\C$ is
a complex-irreducible $\frakk$-module, too, according to Lemma~\ref{le:irreducible}.
Thus already $V$ is necessarily an irreducible $\frakk$-module.

For~(b): Without loss of generality, we can assume that $M_1$ is Hermitian but
not $M_2$\,; let $J_1$ denote the complex structure of $V_1$\,.  
Then, according to Lemma~\ref{le:Komplexifikation}~(b), $\frakk_2$ acts irreducibly on
$V_2^\C$\,, whereas Eq.~\eqref{eq:Vpm} gives the decomposition $V_1^\C=(V_1)_{+\i}\oplus (V_1)_{-\i}$
into two invariant subspaces such that $(V_1)_{+\i}$ is an irreducible
$\frakk_1$-modules (over $\C$). Then $\frakk_2$ and
$\tilde\frakk_1:=[\frakk_1,\frakk_1]$ both are semisimple Lie algebras 
and $\frakk_1=\frakc_1\oplus\tilde\frakk_1$ holds 
according to Lemma~\ref{le:Hermitian}~(c), where $\frakc_1$ denotes the
center of $\frakk_1$\,. Furthermore, we have $\rho_1(\frakc_1)=\R\,J_1$ because of Lemma~\ref{le:Hermitian}~(d).
In particular, $(V_1)_{+\i}$ is even a complex-irreducible $\tilde\frakk_1$-module (since $\frakc_1$ acts through scalars on
$(V_1)_{+\i}$). Therefore, we can apply Lemma~\ref{le:irreducible} to conclude
that $(V_1)_{+\i}\otimes V_2^\C$ is a complex-irreducible
$\frakk$-module. Furthermore, in accordance with Eq.~\eqref{eq:IsomorphismOverC1},
$$
V\to (V_1)_{+\i}\otimes V_2^\C,\; v_1\otimes v_2\mapsto 1/2\,(v_1-\i\,J_1\,v_1)\otimes
v_2$$
is an isomorphism of $\frakk$-modules over $\C$ (where $V$ is seen as a
complex space by means of $J_1$).
Therefore, $V$ is also irreducible over $\C$\,. But $\rho_1(\frakc_1)=\R\,J_1$ holds\,, 
hence $V$ is irreducible over $\R$\,, too. This finishes the
proof for~(b).

For~(c): Here we obtain the decomposition $V_i^\C=(V_i)_{+\i}\oplus (V_i)_{-\i}$ 
into $\frakk_i$-invariant subspaces such that
$\frakk_i$ acts irreducibly on both $(V_i)_{+\i}$ and $(V_i)_{-\i}$ (over $\C$) 
for $i=1,2$\,. Furthermore, using once again Eq.~\eqref{eq:IsomorphismOverC1},$$
V_+\to (V_1)_{+\i}\otimes (V_2)_{+\i},\; (v_1\otimes v_2-J_1\,v_1\otimes J_2\,v_2)\mapsto 
1/2\,\big((v_1-\i\,J_1\,v_1)\otimes (v_2-\i\,J_2\,v_2)\big)
$$
is an isomorphism of $\frakk$-modules over $\C$ (where $V_+$ is seen as a complex space similar as before).
Using similar arguments as before, we now conclude that $V_+$ is an irreducible $\frakk$-module.
Analogously, we can show that $V_-$ is an irreducible
$\frakk$-module, too. Furthermore,
$$
V_+\to V_-, v_1\otimes v_2-J_1\,v_1\otimes
J_2\,v_2\mapsto v_1\otimes v_2+J_1\,v_1\otimes J_2\,v_2
$$
is a linear isomorphism. This finishes the proof for~(c).
 \end{proof}
\end{appendix}

\vspace{2cm}
\begin{center}

 \qquad
 \parbox{60mm}{Tillmann Jentsch\\
  Mathematisches Institut\\
  Universit{\"a}t zu K{\"o}ln\\
  Weyertal 86-90\\
  D-50931 K{\"o}ln, Germany\\[1mm]
  \texttt{tjentsch@math.uni-koeln.de}}

\end{center}

\end{document}